\documentclass[11pt]{article}
\usepackage{epsfig,epsf,fancybox}
\usepackage{amsmath,bm}
\usepackage{mathrsfs}
\usepackage{amssymb}
\usepackage{graphicx}
\usepackage{color}
\usepackage{tabularx}
\usepackage{hyperref, url}
\usepackage{caption}
\usepackage{booktabs}
\usepackage{subcaption}
\newtheorem{assumption}{Assumption}

\newtheorem{theorem}{Theorem}
\newtheorem{lemma}[theorem]{Lemma}
\newtheorem{corollary}[theorem]{Corollary}
\newtheorem{proposition}[theorem]{Proposition}

\newtheorem{example}{Example}

\newcommand{\XX}{{\cal X}}

\newcommand{\RR}{\mathbf R}

\newcommand{\KK}{{\cal K}}

\newcommand{\proof}{{\sc Proof. }}

\newcommand{\cE}{{\cal E}}
\newcommand{\cA}{{\cal A}}
\newcommand{\cL}{{\cal L}}
 
\newcommand{\cX}{\mathcal{X}}

\newcommand{\Rmnum}[1]{\uppercase\expandafter{\romannumeral #1}} 

\usepackage{mathrsfs} %math-script format, then use \mathscr{F}
\usepackage{enumerate} % change the symbols in enumerate
\graphicspath{{figures/}}

\usepackage{algorithm}
\usepackage{algorithmicx}
\usepackage{multirow}
\usepackage{amsmath}
\usepackage{stmaryrd}
\usepackage{float}
\usepackage{graphicx}

\usepackage{xcolor}
\usepackage{algpseudocode}
\DeclareMathOperator*{\argmin}{argmin}
 
\usepackage{color}

\begin{document}

\title{Complexity Analysis of Convex Majorization Schemes \\ for Nonconvex Constrained Optimization}

\author{Nuozhou Wang\thanks{Department of Industrial \& Systems Engineering, University of Minnesota
  (\texttt{wang9886@umn.edu}).}
\and Junyu Zhang\thanks{Department of Industrial Systems Engineering \& Management, National University of Singapore
  (\texttt{junyuz@nus.edu.sg}). }
\and Shuzhong Zhang\thanks{Department of Industrial \& Systems Engineering, University of Minnesota
  (\texttt{zhangs@umn.edu}).}
}

\date{\today}

\maketitle

\abstract{In this paper, we introduce and study various algorithms for solving nonconvex minimization with inequality constraints, based on the construction of convex surrogate envelopes that majorize the objective and the constraints. In the case where the objective and constraint functions are gradient H\"{o}lderian continuous, the surrogate functions can be readily constructed and the solution method can be efficiently implemented. The surrogate envelopes are extended to the settings where the second-order information is available, and the convex subproblems are further represented by Dikin ellipsoids using the self-concordance of the convex surrogate constraints. Iteration complexities have been developed for both convex and nonconvex optimization models. The numerical results show promising potential of the proposed approaches.}

%%================================%%
%% Sample for structured abstract %%
%%================================%%

%\keywords{nonconvex constraint, H\"{o}derian continuity, convex surrogates, complexity}



\maketitle

\section{Introduction}
In this paper, we consider a generic constrained continuous optimization model
\[
\begin{array}{lll}
	(P) & \min & F(x)=f(x)+r(x)\\
	& \mbox{\rm s.t.} & c_{i}(x)\leq 0,\, i=1,2,...,m \\
	& & x \in \XX
\end{array}
\]
where $\XX$ is a closed convex subset of $\RR^n$, the functions $f(x)$ and $c_i(x)$'s are smooth but possibly nonconvex, and the function $r(x)$ is a convex and lower semicontinuous function but possibly nonsmooth. Known as nonlinear programming, there is a long history and a large body of literature for the problem $(P)$, whose modeling power is certainly beyond doubt. As an example, computing the stability number for a given graph is a copositive program, see \cite{D2010} for a survey, which can be posed as $(P)$. Moreover, Burer \cite{burer2009copositive} showed that many difficult combinatorial optimization problems such as quadratic assignment problems can also be posed as copositive programming, such as: 
\[
\begin{array}{lll}
	& \min & \mathrm{Tr}(Y^\top C Y) \\
	& \mbox{s.t.} & \mathrm{Tr}(Y^\top \!A_iY)   \le b_i, \, i=1,2,...,m \\
	&  & Y \ge 0. 
\end{array}
\]
As a test instance of this paper, the above nonconvex formulation of copositive programming is a special case of $(P)$, where both the objective function and constraints can be nonconvex. Its success depends critically on the ability to find high quality (approximate) feasible solutions on an empirical basis, even though the globally optimal solutions are out of reach in the worst cases.

The difficulties notwithstanding, there has been a considerable amount of research attention on the nonconvex optimization model $(P)$.  One classic approach is to leverage on the augmented Lagrangian method (ALM) and its variants \cite[etc.]{bertsekas1997nonlinear,nocedal1999numerical,bertsekas2014constrained,grapiglia2021complexity,ZZZ24}, one may refer to \cite{deng2025augmented} for a more recent and complete survey. However, a typical issue of the ALM-based approach is that, when the nonconvex subproblems are only solved to first-order stationarity, one cannot rule out the possibility of converging to some stationary points of a certain violation penalty function that are in general infeasible to the original problem; see e.g.\ \cite[Corollary 3.7]{de2023constrained}. To circumvent the infeasibility issue, which can occur regardless of the constraint qualifications, one common remedy is to additionally assume certain generalized quadratic gradient-dominance condition on the quadratic penalty function, which directly implies the feasibility of the stationary points. For example, by assuming a global Polyak-Lojasiewicz type condition, \cite{sahin2019inexact} established an $\mathcal{O}(\epsilon^{-4})$ complexity of an inexact ALM algorithm for problem $(P)$ with general nonconvex functional constraints. Given a similar regularity condition, \cite{li2021rate} derived an improved complexity of $\mathcal{O}(\epsilon^{-3})$, while considering a simpler problem setting that only has functional equality constraints. 

Apart from making strong regularity conditions, the penalty approach is a popular alternative; see e.g.\ \cite{cartis2011evaluation,cartis2014complexity}. For example, \cite{lin2022complexity} considers an inexact proximal point penalty approach. Assuming the boundedness of the objective and constraint functions, they are able to preserve the strong convexity of the subproblem sequence. By letting the penalty parameter to increase to $+\infty$, 
the constraint violation becomes controllable. Then several iteration complexities ranging from $\mathcal{O}(\epsilon^{-2.5})$ to $\mathcal{O}(\epsilon^{-4})$ were established under different settings. However, without driving the penalty parameter to $+\infty$, a general penalty approach does not necessarily rule out the possibility of converging to an infeasible solution either, unless some additional assumptions on the iterative sequence are made a priori. 

Independent of the ALM or penalty approaches, in \cite{boob2023stochastic} the authors developed a Constraint Extrapolation (ConEx) method for stochastic optimization with functional constraints. In each iteration, ConEx solves a quadratically regularized proximal point subproblem (without linearizing the objective and constraint functions), under a certain MFCQ-type assumption, they showed that the iteration complexity of their algorithm is $O(1/\epsilon)$ for the convex case and $O(1/\epsilon^2)$ for the nonconvex case. As a consequence of stochasticity, the authors did not discuss the feasibility of the iterates.  

\textcolor{black}{If one insists on feasibility guarantees, there is a line of research that requires the constraints to be satisfied at every iteration. A classical method of this type is the moving balls approximation method \cite{auslender2010moving}. Under Lipschitz continuity of the gradients, each nonlinear constraint is replaced by a ball-shape inner approximation. Given appropriate constraint qualifications, asymptotic (subsequence) convergence to KKT points are also established in \cite{auslender2010moving}. For the nonconvex case, if in addition that the objective and constraint functions are tame functions, the full sequence converges to some KKT point with linear or sublinear local rate can be established for moving ball methods by KL-type analyses \cite{bolte2016majorization}. For the convex case, an $O(1/\epsilon)$ iteration complexity of the moving ball method is established by \cite{bolte2017multiproximal} under appropriate Lipschitz smoothness conditions. More recently, the feasibility requirement has received renewed attention under the name of safe optimization, motivated by applications such as  }  power system \cite{guo2023safe,guo2023safenon,guo2023safeLP}, constrained reinforcement learning \cite{chow2019lyapunov,zhao2023state}, and safe Bayesian optimization \cite{fiedler2024safety}, 
where safe interactions/experiments, which naturally translate to feasible iterates,  with the environment is favorable. In particular, our main formulation $(P)$ is considered in \cite{guo2023safe,guo2023safenon,guo2023safeLP}, among which  \cite{guo2023safeLP} solves linear programs (LP) at each iteration to get a safe solution while an iteration complexity analysis is missing; \cite{guo2023safe,guo2023safenon} solve a sequence of quadratically constrained quadratic programming (QCQP) subproblems. Under the gradient Lipschitz conditions for both objective and constraint functions, zeroth-order algorithms and their sample complexities have been discussed in the two papers, for convex and nonconvex problems, respectively. Under similar ideas, \cite{nutalapati2019constrained}  proposed to use general surrogate functions in the subproblems to ensure the feasibility of the iterates while tackling stochasticity (sampling from finite-sum format) using variance reduction techniques. Neither the convergence rate nor the sample complexity was discussed in this paper.

It is worth mentioning that in the aforementioned works, finite-step and non-asymptotic complexity results for the nonconvex model $(P)$ are mostly restricted to the settings that assume the objective and constraint functions to have gradient Lipschitz continuity property. In this paper, we attempt to break away from such an assumption, treating it merely as a starting point for our new approach. Specifically, we shall study scenarios where neither convexity nor gradient Lipschitz properties hold, as long as a convex envelope is available. As an example, in linear regression estimation based on $\ell_p$-loss functions where $1<p<2$, the gradient Lipschitz property fails to hold; see e.g.\ \cite{BGLSS21} for a discussion on why the case for $1<p<2$ may be statistically significant. The regression model with an $\ell_p$ loss function ($1<p<2$) corresponds to the gradient H\"olderian continuous case to be studied in Section~\ref{holderian_case}. In general, a combination of non-convexity and non-Lipschitz gradient properties poses a challenge to most existing solution methods.  

\vspace{0.2cm}

\textbf{\large Our contribution and organization.}
In Section \ref{GeneralSchemes}, we propose a generalized framework called Convex Majorization Minimization Approach (CMMA) to solve the general problem $(P)$ with nonconvex objective function and nonconvex functional constraints. It works with the available convex envelopes of the objective and the constraint functions that are constructed depending on the structure of the problems at hand. We shall mention upfront that our approach mainly differs from most other approaches in that we do not rely solely on the gradient Lipschitz continuity of the objective and constraint functions though the latter could be helpful, and that we insist on having a feasible solution as output. Hence, CMMA also belongs to the safe optimization algorithms. Under gradient H\"olderian assumption and appropriate regularity conditions, roughly, an $\mathcal{O}(\epsilon^{-(1+ {1}/{\kappa})})$ iteration complexity for finding $\epsilon$-approximate KKT points has been established, where $\kappa$ refers to the H\"older exponent. We would like to highlight that the regularity condition of this paper (Assumption \ref{assum:upper}) is a very mild condition, and it can be implied by a number of different alternative conditions discussed in Section \ref{GeneralSchemes} and \ref{Surrogate_Construction}.  Then in Section \ref{holderian_case}, if in addition that the problems possess convexity, an improved $\mathcal{O}(\epsilon^{-1/\kappa})$ complexity for finding an $\epsilon$-approximate optimal solution can be obtained for the GHMA, an instantiation of the CMMA with convex problems. Faster linear convergence has also been provided given strong convexity.

Note that the above iteration complexities of CMMA and GHMA are based on the access to efficient subproblem oracles. In Section \ref{Solving_Subproblem}, we develop an efficient primal-dual method for solving the typical subproblems under the Lipschitz or H\"olderian majorization construction. In case the number of constraints is too large and solving the original subproblem remains expensive, Section \ref{Barrier_Dikin} provides two lightweight alternatives: (i) the CEB method that penalizes the surrogate constraints by self-concordant barriers; (ii) the CEAS method that approximates many (quadratic) surrogate constraints by one single Dikin ellipsoid constraint. As a trade-off for simpler subproblems, CEB and CEAS exhibit weaker worst-case iteration complexity guarantees.  

Finally, we provide a brief discussion on a potential second-order extension of our method in Section \ref{Second_Order_Extension}, and we present the numerical experiments in Section \ref{numerical}.

\vspace{0.2cm}

\textbf{\large Related works.}  In this paragraph, we would like to discuss a couple of closely related works to our paper. In \cite[Chapter 7]{CP2022}, a so-called \textit{method by surrogation} is studied, under which the CMMA framework proposed in this paper also belongs. The authors focus on overcoming difficulties arising from general nonsmoothness while providing limited asymptotic convergence analysis with no explicit rates. Our focus, however, is for problems with limited smoothness and we aim at provable computational complexities. \textcolor{black}{For the same reason, our work also differs from the current results on moving ball methods \cite[etc.]{auslender2010moving,bolte2016majorization,bolte2017multiproximal} as they only corresponds to a special case of CMMA, while we provide finite step convergence guarantees beyond convexity and Lipschitz smoothness.}
In \cite{nutalapati2019constrained}, the authors proposed to use surrogate functions in the subproblems to ensure the feasibility of the iterates while tackling stochasticity (sampling from finite-sum format) using variance reduction techniques. Their algorithm shares a similar algorithmic form of CMMA while differing in the regularity conditions and analytic requirements on the surrogates. They showed convergence of their method while an analysis of convergence rate is missing. In \cite{NN2024}, the authors considered high-order Taylor expansions as surrogates which are nonconvex in general, and studied convergence of this scheme under various assumptions, while the focus of the current paper is on using convex surrogate subproblems. 

\vspace{0.2cm}

\textbf{\large Notation.}  For any integer $m\geq1$, we use $[m]$ to denote the set of integers $\{1,\cdots,m\}$. For any matrix $X$, we use $\sigma_{\min}(X)$ and $\sigma_{\max}(X)$ to denote the minimum and maximum singular values of $X$, respectively. For any matrix $X$, we default $\|X\|$ its spectral norm. For any  vector $x$, we use $\|x\|$ to default the standard $\ell_2$-norm, and we use $\|x\|_\infty$ to denote the $\ell_\infty$-norm. Given a set $\mathcal{X}\subseteq\RR^n$ and a point $x\in\RR^n$, we use $\mathrm{dist}(x,\mathcal{X}):=\min_{z\in\mathcal{X}}\|z-x\|$ to denote the distance between $x$ and $\mathcal{X}$, and we use $\mathrm{diam}(\mathcal{X}):=\max_{y,z\in\mathcal{X}}\|y-z\|$ to denote the diameter of the set $\mathcal{X}$.

\section{ General Solution Scheme Via Convex Surrogate } \label{GeneralSchemes}
In this section, assuming the access to the upper surrogates of the objective and constraint functions that satisfy Assumption~\ref{assumption:surrogate}, we introduce the general framework of this paper and establish its oracle complexity based on how many (convex) subproblems are solved. Typically, we will assume by default that the structure of the nonsmooth term $r$ is simple enough so that the subproblems can be efficiently solved. The construction of qualified surrogates is discussed in the next section under several different context. 

\begin{assumption} \label{assumption:surrogate}
	For problem $(P)$, given any feasible $y$, there exist convex and continuously differentiable surrogate functions $\tilde{f}(\cdot\mid y)$ and $\tilde{c_i}(\cdot\mid y)$ 
	such that for any $x$ it holds that 
	\[
	\tilde{f}(x\mid y)\geq f(x), \quad \tilde{c_i}(x\mid y)\geq c_i(x), \quad i\in[m],
	\] 
	where the equality is attained 
	if $x=y$. In addition, we assume $\nabla \tilde{f}(\cdot\mid y)$ and $\nabla \tilde{c_i}(\cdot\mid y), i\in[m]$ to satisfy the following H\"{o}lderian continuous property: 
	\[
	\begin{array}{l}
		\|\nabla\tilde{f}(x\mid y)-\nabla f(x)\|\le 2L\|x-y\|^{\kappa} , \quad \|\nabla\tilde{c_i}(x\mid y)-\nabla c_i(x)\|\le 2L_i\|x-y\|^{\kappa_i}, \quad i\in[m],
	\end{array}
	\]
	for some exponents $\kappa,\kappa_i\in(0,1]$ and constants $L,L_i>0$, $i\in {[m]}$.
	We also assume there exist $N_i>0,i\in[m]$ s.t. $\|\nabla\tilde{c_i}(x\mid y)\|\le N_i, i\in[m]$ for all $x,y$ that are feasible to $(P)$. 
\end{assumption}

In particular, we assume the constants $\kappa,\kappa_i\in(0,1]$ and $L,L_i>0$, $i\in[m]$, are known constants to users. When these constants are  unknown, efficient heuristics are required to estimate them as the input to the algorithm, or some adaptive procedures without such knowledge need to be developed, which are interesting future research topics. Besides, we assume Slater's condition for problem $(P)$ in our setting.  
\begin{assumption} \label{assumption:Slater}
	An optimal solution exists for $(P)$, and there is a strictly feasible solution $\hat{x}\in\XX$ s.t. $c_i(\hat{x}) < 0$, $i\in[m]$. We also assume the problem to have bounded level set. 
\end{assumption}

Based on Assumptions \ref{assumption:surrogate} and \ref{assumption:Slater}, the following convex majorization minimization approach naturally arises as a generic algorithmic framework.  \vspace{0.6cm}

\shadowbox{\begin{minipage}{4.8in}
		{\bf Algorithm CMMA (Convex Majorization Minimization Approach)}
		
		\begin{description}
			
			\item[Step 0:] Choose $x^0$ to satisfy $c_i(x^0) < 0$, $i\in [m]$, and let $k:=0$.
			
			\item[Step 1:] Solve the following subproblem with $x = x^k$ 
			\[
			\begin{array}{lll}
				(P_{x}) \quad z^*(x):=& \argmin_{z} & \tilde{f}(z\mid x) + \frac{L}{\kappa+1}\|z-x\|^{\kappa+1}+r(z) \qquad\qquad\qquad\quad\\
				& \,\,\,\,\,\mbox{ s.t.}  & \tilde{c_i}(z\mid x)\le 0 ,\,\, i\in[m],\\
				& & z \in \XX
			\end{array}
			\]
			and update the next iterate as $x^{k+1} = z^*(x^k)$.
			
			\item[Step 2:] If $x^{k+1}=x^k$ then stop. Otherwise, let $k:=k+1$, and go to {\bf Step 1}.
			
		\end{description}
		
\end{minipage}}

\subsection{Iteration Complexity of Algorithm CMMA}
Denote $d^k=x^{k+1}-x^k$ and $\Delta=F(x^0)-F(x^*)$, where $x^*$ is an optimal solution of $(P)$. According to Assumption \ref{assumption:Slater}, we have $\Delta\geq 0$. We observe that 
\begin{eqnarray*}
	F(x^{k+1}) &\leq& \tilde{f}(x^{k+1}\,|\,x^k)+r(x^{k+1}) \\
	&\leq &\tilde{f}(x^{k}\,|\, x^k)+r(x^{k})-\frac{L\|d^k\|^{\kappa+1}}{\kappa+1} \\
	&=& F(x^k)-\frac{L\|d^k\|^{\kappa+1}}{\kappa+1},
\end{eqnarray*}
where the first inequality is according to Assumption \ref{assumption:surrogate} and the second inequality is because $x^{k+1}$ optimizes the subproblem $(P_{x^k})$. Telescoping the above inequalities gives 
$$\sum_{k=0}^{K-1}\frac{L}{\kappa+1}\|d^k\|^{\kappa+1}\leq F(x^0)-F^*=\Delta, $$
which further implies that 
\begin{equation}
	\label{eqn:dkmin}
	\min_{0\leq k\leq K-1} \|d^k\|\leq \left(\frac{(\kappa+1)\Delta}{LK}\right)^{\frac{1}{\kappa+1}}.
\end{equation}
That is, the algorithmic residual $\|d^k\|\to0$ at an $O(k^{-\frac{1}{\kappa+1}})$ sublinear rate. To connect the convergence rate of $\|d^k\|$ with the convergence rate to a KKT point, we introduce the following assumption, where we denote 
${\rm Lev}(x^0):=\{x\in\XX:F(x)\leq F(x^0), c_i(x)\leq 0, i\in[m]\}$ the level set of any feasible solution $x^0$ to problem $(P)$. 
\begin{assumption}  
	\label{assum:upper}
	There exist $B,\varrho>0$ such that for any subproblem $(P_x)$ with $x\in{\rm Lev}(x^0)$ being strictly feasible, then an associated Lagrangian multiplier $\lambda\in\RR^m_+$ exists which satisfies $\|\lambda\|_\infty\leq B$ as long as $\|z^*(x)-x\|\leq \varrho$. (Note: $z^*(x)$ is defined in Step 1 of Algorithm CMMA.)
\end{assumption}

In the context of Algorithm CMMA, this assumption essentially guarantees the uniform boundedness of the Lagrangian multipliers associated with the subproblems, given that the algorithm is close to being stationary, that is, $\|x^{k+1}-x^{k}\|\leq\varrho$ is sufficiently small for some threshold $\varrho>0$. It rules out the degenerate situations where the multipliers explode. Remark that this type of bounded multiplier assumption is quite standard in sequential-subproblem frameworks. Though, unlike Assumptions \ref{assumption:surrogate} and \ref{assumption:Slater}, Assumption~\ref{assum:upper} cannot be verified {\it a priori}. However, it can be implied by a number of other commonly used alternatives, as shall be illustrated in sections~\ref{subsec:about_dual-boundedness}, \ref{subsection:Holder-NCVX}, and \ref{holderian_case}. 
Denoting the Lagrangian function of $(P)$ as 
$\cL(x,\lambda):= f(x) + r(x) + \sum_{i=1}^m\lambda_i c_i(x)$, we present the following theorem to characterize the convergence rate of CMMA to approximate KKT points of $(P)$. 
\begin{theorem} 
	\label{iter-compl}
	Suppose Assumptions \ref{assumption:surrogate}, \ref{assumption:Slater}, and \ref{assum:upper} hold, and we run CMMA for $K\geq \frac{(\kappa+1)\Delta}{L\varrho^{\kappa+1}}$ iterations. 
	Then, for some $0\leq k_*\leq K-1$, with $\lambda\in\RR^m_+$ being the Lagrangian multiplier of $(P_{x^{k_*}})$, we have 
	\begin{eqnarray*}
		{\rm dist}\Big(0,\partial \cL (x^{k_*+1},\lambda)+{\cal N}_{\XX}(x^{k_*+1})\Big) \leq  O\Big(K^{-\kappa_{\min}/(\kappa+1)}\Big),
	\end{eqnarray*}
	where $\kappa_{\min} = \min\{\kappa,\kappa_1,\cdots,\kappa_m\}$. For any $i\in[m]$, if $\lambda_i>0$,  then 
	$$0\leq -c_i(x^{k_*+1}) \leq O\big(K^{-1/(\kappa+1)}\big).$$
\end{theorem}
As a remark, in case $\kappa=\kappa_1=\cdots=\kappa_m$ and the target precision $\epsilon$ is small enough, then the iteration complexity of CMMA to reach an $\epsilon$-KKT solution is $O\big(\epsilon^{-(\kappa+1)/\kappa}\big)$. 

\begin{proof}
	According to \eqref{eqn:dkmin}, when the iteration number $K$ is large enough so that it satisfies the requirement of Theorem \ref{iter-compl}, there exists some $0\leq k_*\leq K-1$ such that $d^{k_*} = x^{k_*+1}-x^{k_*}$ satisfies
	$$\|d^{k_*}\|\leq \left(\frac{(\kappa+1)\Delta}{LK}\right)^{\frac{1}{\kappa+1}} \leq \varrho.$$
	
	By Assumption~\ref{assum:upper} for the  subproblem in Algorithm CMMA, there exist an $l^{k_*+1}\in \partial r(x^{k_*+1})$ and a Lagrangian multiplier $\lambda\in\RR^m_+$ such that 
	\begin{eqnarray}
		\label{eqn:sub-KKT} 
		\begin{cases}
			\!\nabla \tilde{f}(x^{k_*+1}\mid x^{k_*})\!+\!l^{k_*+1}\!+\!L \|d^{k_*}\|^{\kappa-1}d^{k_*}+\sum_{i=1}^m \lambda_i\nabla\tilde{c_i}(x^{k_*+1}\!\mid \!x^{k_*})\!\in\! - {\cal N}_{\XX}(x^{k_*+1}), \\
			\lambda_i\tilde{c_i}(x^{k_*+1}\mid x^{k_*})=0, \quad i\in[m],
		\end{cases}
	\end{eqnarray}
	where ${\cal N}_{\XX}(x)$ denotes the normal cone of $\XX$ at $x$. Then, by Assumption \ref{assumption:surrogate} we have
	\begin{eqnarray*}
		& & {\rm dist}\Big(0,\partial \cL (x^{k_*+1},\lambda)+{\cal N}_{\XX}(x^{k_*+1})\Big)\\
		&\leq& {\rm dist}\Big(0,\nabla f(x^{k_*+1})+l^{k_*+1}+\sum_{i=1}^m \lambda_i\nabla c_i(x^{k_*+1})+{\cal N}_{\XX}(x^{k_*+1})\Big) \\
		&\overset{(i)}{\leq}&{\rm dist}\Big(0,\nabla \tilde{f}(x^{k_*+1}\mid x^{k_*})+l^{k_*+1}+\sum_{i=1}^m \lambda_i\nabla\tilde{c_i}(x^{k_*+1}\mid x^{k_*})+{\cal N}_{\XX}(x^{k_*+1})\Big)\\
		& & +2L\|d^{k_*}\|^{\kappa}+2\sum_{i=1}^m \lambda_i L_i\|d^{k_*}\|^{\kappa_i} \\
		&\overset{(ii)}{\le }&
		3L\|d^{k_*}\|^{\kappa}+2\sum_{i=1}^m \lambda_i L_i\|d^{k_*}\|^{\kappa_i}\nonumber,
	\end{eqnarray*} 
	where (i) is due to Assumption \ref{assumption:surrogate} and (ii) is due to the first relationship of \eqref{eqn:sub-KKT}. 
	Together with Assumption \ref{assum:upper} that suggests $\|\lambda\|_\infty\leq B$ and the bound on $\|d^{k_*}\|$, we obtain 
	\begin{eqnarray*}
		&&{\rm dist}\Big(0,\partial \cL (x^{k_*+1},\lambda)+{\cal N}_{\XX}(x^{k_*+1})\Big)\\
		&\leq& 3L\left(\frac{(\kappa+1)\Delta}{LK}\right)^{\frac{\kappa}{\kappa+1}}+2\sum_{i=1}^m BL_i\left(\frac{(\kappa+1)\Delta}{LK}\right)^{\frac{\kappa_i}{\kappa+1}},
	\end{eqnarray*} 
	which is an $O\big(K^{-\kappa_{\min}/(\kappa+1)}\big)$ upper bound. In addition, for each $i\in[m]$, if $\lambda_i>0$, the second relationship of \eqref{eqn:sub-KKT} implies that $\tilde{c_i}(x^{k_*+1}\mid x^{k_*})=0$, which further leads to $$0\leq -c_i(x^{k_*})=\tilde{c_i}(x^{k_*+1}\mid x^{k_*})-\tilde{c_i}(x^{k_*}\mid x^{k_*}) \overset{(i)}{\leq} N_i\|d^{k_*}\| \leq   N_i \left(\frac{(\kappa+1)\Delta}{LK}\right)^{1/(\kappa+1)},$$
	which is an $O\big(K^{-1/(\kappa+1)}\big)$ upper bound, where (i) is due to Assumption \ref{assumption:surrogate}.
\end{proof}

\subsection{Discussions on Assumption \ref{assum:upper}} \label{subsec:about_dual-boundedness}

Though Assumption \ref{assum:upper} is less conventional compared to Assumptions~\ref{assumption:surrogate} and \ref{assumption:Slater}, there are many other tangible conditions that imply this assumption. Below we provide such a mild condition that implies Assumption \ref{assum:upper}, which requires each $(P_k)$
to be strictly feasible in a uniform fashion. 
\begin{proposition}
	\label{proposition: As.4-2-As.3}
	If there exist $U,\tau>0$ such that for any subproblem $(P_x)$ with $x\in{\rm Lev}(x^0)$ being strictly feasible to $(P)$, there is (assumed to exist) a point $\hat{x}$ strictly feasible to subproblem $(P_x)$ that satisfies
	$$\tilde{f}(\hat{x}\mid x)+r(\hat{x})+\frac{L}{1+\kappa}\|\hat{x}-x\|^{1+\kappa}\leq U\quad\mbox{ and }\quad\tilde{c}_i(\hat{x} \mid x) \leq -\tau, \,\,\,i\in [m],$$
	then Assumption~\ref{assum:upper} will hold. 
\end{proposition}

Note that this is a stronger version of the Slater condition by assuming uniform bounds with respect to all the interior points. This type of conditions often come in handy when
Lagrangian multipliers need to be bounded; see e.g.~\cite{bai2022achieving,chen2022near}. Below we present a proof for this proposition.

\begin{proof}
	Note that the coercivity of the term $\frac{L}{\kappa+1}\|z-x\|^{\kappa+1}$ implies the existence of an optimal solution $z^*(x)$ to the subproblem $(P_x)$, together with the Slater's condition implied by the existence of $\hat{x}$, there exist  Lagrangian multiplier $\lambda\in\RR^m_+$ and a subgradient $l\in \partial r(z^*(x))$ such that
	\[
	\Big\langle \nabla \tilde{f}(z^*(x)\mid x)+l+L\|d\|^{\kappa-1}d+\sum_{i=1}^m \lambda_i\nabla\tilde{c_i}(z^*(x)\mid x)\,,\, \hat{x}-z^*(x) \Big\rangle\geq 0, 
	\]
	where $d = z^*(x)-x$ and $\hat{x}, z^\star(x) \in \XX$. By convexity of $\tilde{f}(\cdot\mid x)$, $r(\cdot)$, $\frac{L}{1+\kappa}\|x-x^k\|^{1+\kappa}$ and $\tilde{c_i}(\cdot\mid x^k)$, we further have
	\begin{eqnarray}
		0&\leq &\tilde{f}(\hat{x}\mid x)+r(\hat{x})+\frac{L}{1+\kappa}\|\hat{x}-x\|^{1+\kappa}+\sum_{i=1}^m \lambda_i \tilde{c_i}(\hat{x}\mid x)\nonumber\\
		&&-\tilde{f}(z^*(x)\mid x)-r(z^*(x))-\frac{L}{1+\kappa}\|d\|^{1+\kappa}-\sum_{i=1}^m \lambda_i \tilde{c_i}(z^*(x)\mid x)\nonumber\\
		&\leq& \tilde{f}(\hat{x}\mid x)+r(\hat{x})+\frac{L}{1+\kappa}\|\hat{x}-x\|^{1+\kappa}+\sum_{i=1}^m \lambda_i \tilde{c_i}(\hat{x}^k\mid x)-F(z^*(x)) \nonumber\\
		&\leq& U+\sum_{i=1}^m \lambda_i (-\tau)-F^*, \nonumber
	\end{eqnarray}
	where $F^*$ denotes the optimal value and the second inequality is because 
	$\sum_{i=1}^m \lambda_i \tilde{c_i}(z^*(x)\mid x)=0.$
	Therefore, we have $\lambda_j\leq \sum_{i=1}^m \lambda_i \leq \frac{U-F^*}{\tau}=:B, \forall j\in[m]$. That is, Assumption \ref{assum:upper} holds with $\varrho=+\infty$ and $B = \frac{U-F^*}{\tau}.$
\end{proof}

We would like to emphasize that the condition described in Proposition \ref{proposition: As.4-2-As.3} is not the only one guaranteeing Assumption \ref{assum:upper}. In the later discussion, we will introduce other sufficient conditions that imply  Assumption~\ref{assum:upper}, based on the detailed constructions of the surrogate functions, as well as the properties of the problem itself.

\section{Constructing Convex Surrogates} 
\label{Surrogate_Construction}
\subsection{Gradient H\"olderian Continuous Nonconvex Function}
\label{subsection:Holder-NCVX}
Let us first consider the case where $f(x)$ and $c_i(x)$ are smooth but possibly nonconvex functions. 
We also assume $\XX = \RR^n$ to be the entire space and we assume the gradients of these functions are H\"olderian continuous.  
To simplify the presentation, we adopt the following terminology that a mapping $\psi$ is $(L,\kappa)$-H\"olderian continuous if $\|\psi(x)-\psi(y)\|\leq L\|x-y\|^\kappa$ for any $x,y$. Based on this notation, we proceed to a detailed study of upper bounding in this section. 
\begin{assumption}
	\label{assumption:Holder-Grad} 
	The function $f$ is $(L,\kappa)$-H\"olderian continuous for some $L>0,\kappa\in(0,1]$, and the function $c_i$ is $(L_i,\kappa_i)$-H\"olderian continuous for some $L_i>0,\kappa_i\in(0,1]$, for each $i\in[m]$. 
\end{assumption}
Under this assumption, for any $x,y$, we have 
\begin{equation}
	\label{eqn:descent-lemma-Holder}
	\begin{cases}
		f(x) \leq f(y) + \nabla f(y)^\top (x-y ) + \frac{L}{1+\kappa} \|x-y\|^{1+\kappa},\\
		c_i(x) \leq c_i(y) + \nabla c_i(y)^\top (x -y) + \frac{L_i}{1+\kappa_i} \|x-y\|^{1+\kappa_i}, i\in[m].
	\end{cases}
\end{equation} 
For convenience of our discussion, we will take $L_i$ to be strictly greater than the minimum constant. That is, we set $L_i$ to be strictly greater than $\sup_{x\neq y}\frac{\|\nabla c_i(x)-\nabla c_i(y)\|}{\|x-y\|^{\kappa_i}}$, for each $i\in[m]$. 
Under Assumption~\ref{assumption:Holder-Grad}, a natural construction of the upper convex surrogates can be: 
\begin{equation}
	\label{eqn:Holder-surrogate}
	\begin{cases}
		\tilde{f}(x\mid y) \,:= f(y)\,+\langle \nabla f(y), x-y \rangle\,+\,\frac{L}{1+\kappa}\|x-y\|^{1+\kappa}, & \\
		\tilde{c_i}(x\mid y) := c_i(y)+\langle \nabla c_i(y), x-y \rangle+\frac{L_i}{1+\kappa_i}\|x-y\|^{1+\kappa_i}, & i\in[m].
	\end{cases}
\end{equation}
One can easily verify that the above surrogates satisfy Assumption \ref{assumption:surrogate}.
Because we choose $L_i$ to be strictly greater than the minimal possible constant, we know the strict feasibility of $x^k$ must imply the strict feasibility of $x^{k+1}$. Given a strictly feasible initial solution $x^0$, this inductively ensures the strict feasibility of all the iterations.  Therefore, Slater's condition holds for all the subproblems $(P_k)$, combined with the coercivity ensured by the proximal term, an optimal solution and the associated Lagrangian multipliers always exist. Next, let us present an alternative condition that guarantees Assumption~\ref{assum:upper}.

\begin{proposition} \label{proposition:Gen-LICT-to-Assp3}
	Assumption \ref{assumption:Holder-Grad} together with the following conditions implies Assumption \ref{assum:upper}:\\
	\noindent\textbf{\emph{(i)}.} There exist constants $\delta,\delta'>0$ such that for  $\forall x\in{\rm Lev}(x^0)$, we have $\sigma_{\min}(J_I(x))\geq \delta'$, where the rows of matrix $J_I(x)$ consist of $\nabla c_i(x)^\top$ with $i\in I(x):=\{i\in[m]:c_i(x)\geq-\delta\}$.\\
	\noindent\textbf{\emph{(ii)}.} There exist constants $M,M_i>0, i\in[m]$ such that for  $\forall x\in{\rm Lev}(x^0)$, we have $\|u\|\leq M$ for all $u\in \nabla f(x)+\partial r(x)$, and $\|\nabla c_i(x)\|\leq M_i$ for all $i\in[m]$.  
\end{proposition}

\begin{proof}
	For any subproblem $(P_x)$ with $x\in{\rm Lev}(x^0)$ being strictly feasible to $(P)$, the KKT condition states that there exists an $l\in \partial r(z^*(x))$ and Lagrangian multiplier $\lambda\in\RR^m_+$ such that
	\begin{equation} \label{eqn:gen-LICQ-KKT}
		\begin{cases}
			\nabla f(x)+l+2L\|d\|^{\kappa-1}d+\sum_{i=1}^m \lambda_i\big(\nabla c_i(x)+L_i\|d\|^{\kappa_i-1}d\big)=0, \\
			\lambda_i\tilde{c_i}(z^*(x)\mid x) = 0, 
		\end{cases}
	\end{equation} 
	where $d = z^*(x)-x$. To verify Assumption \ref{assum:upper}, let us suppose $\|d\|\leq \varrho$ 
	for some small enough $\varrho$ s.t.
	\begin{equation}
		\label{eqn:gen-LICQ-varrho}
		\delta \,\,\geq\,\, \max_{i\in[m]} \Big\{M_i\varrho + \frac{L_i
			\varrho^{\kappa_i+1}}{\kappa_i+1}\Big\}.
	\end{equation}
	For $i\notin I(x)$, we have $c_i(x) < -\delta$. By \eqref{eqn:gen-LICQ-varrho}, definition of $\tilde{c_i}$, and condition \textbf{(ii)}, we have
	\begin{eqnarray} 
		\big|\tilde{c_i}(z^*(x)\mid x)-c_i(x)\big| 
		\leq M_i\|d\|+\frac{L_i}{\kappa_i+1}\|d\|^{\kappa_i+1} \leq  \delta.\nonumber
	\end{eqnarray}
	This further implies that $\tilde{c_i}(z^*(x)\mid x)<0$ and hence $\lambda_i=0$.
	
	For any $i\in I(x)$, condition \textbf{(i)} can be equivalently stated as the existence of a unit vector $\|e_i\|=1$, such that $\langle \nabla c_j(x), e_i \rangle=0$, for all $j\in I(x)\setminus \{i\}$ and $\langle \nabla c_i(x), e_i \rangle \geq \delta'$. Consequently, we have 
	\begin{eqnarray}
		\lambda_i\delta'  \leq \lambda_i\langle\nabla c_i(x),e_i\rangle =  -\langle\nabla f(x)+l,e_i\rangle - \Big(2L\|d\|^{\kappa-1}+\sum_{j=1}^m \lambda_j L_j \|d\|^{\kappa_i-1}\Big)\langle d,e_i\rangle,  \nonumber
	\end{eqnarray}
	where the equality is obtained by taking inner product between $e_i$ and the first equation of \eqref{eqn:gen-LICQ-KKT} and then reordering the terms. Then using Cauchy-Schwarz inequality and condition \textbf{(ii)} of Proposition \ref{proposition:Gen-LICT-to-Assp3} to the above inequality gives 
	\begin{equation} \label{ineq:lambdabound}
		\lambda_i\delta' \leq M+2L\varrho^\kappa+\sum_{j=1}^m \lambda_j L_j\varrho^{\kappa_i}, \quad i\in I(x).
	\end{equation} 
	Note that when $i \notin I(x)$,  we have already proved that $\lambda_i=0$. As $\lambda\geq0$, the above inequality further indicates that 
	$\delta'\|\lambda\|_\infty\leq M+ 2L\varrho^\kappa+\sum_{j=1}^mL_j\varrho^{\kappa_i} \|\lambda\|_\infty$. Suppose the $\varrho$ in Assumption \ref{assum:upper} is taken small enough so that 
	\begin{equation}
		\label{eqn:gen-LICQ-varrho'}
		\delta'/2 \geq \sum_{j=1}^mL_j\varrho^{\kappa_i}.
	\end{equation}
	Then we obtain the following upper bound for the multipliers  
	$\|\lambda\|_\infty\leq 2\left(M+ 2L\varrho^\kappa\right)/\delta'.$
	That said, Assumption \ref{assum:upper} holds with any small enough $\varrho$ satisfying \eqref{eqn:gen-LICQ-varrho} and \eqref{eqn:gen-LICQ-varrho'}, and $B = 2\left(M+ 2L\varrho^\kappa\right)/\delta'.$
\end{proof}
%\hfill $\Box$ 

\subsection{Difference of Convex Functions}
In nonconvex optimization, an important class of problem is the optimization of difference of convex functions, where we assume the objective and constraint functions of $(P)$ to possess the following structure: $f(x)=f^1(x)-f^2(x)$ and $c_i(x)=c_i^1(x)-c_i^2(x)$, where $f^1,f^2,c_i^1, c_i^2$ are all convex smooth functions; see e.g.~\cite{CP2022}.
Under this setting, surrogate functions satisfying Assumption \ref{assumption:surrogate} can  be constructed in the following way, where only the subtracted term is linearized. 

\begin{proposition}
	\label{proposition:DC-Assp1}
	Suppose $f^1,f^2,c_i^1,c^2_i, i\in[m]$ are all smooth convex functions,  $\nabla f, \nabla f^1$ are $(L,\kappa)$-H\"oderian continuous, and $\nabla c_i, \nabla c_i^1$ are $(L_i,\kappa_i)$-H\"olderian continuous, $i\in[m],$ for some $L,L_i>0$ and $\kappa,\kappa_i\in(0,1], i\in[m]$. If in addition $\|\nabla c_i^j(x)\|\leq M_i,$ for all $i\in[m], j\in[2]$ and for all feasible solution $x$, then the following construction of surrogate function satisfies Assumption \ref{assumption:surrogate}:
	$$\begin{cases}
		\,\tilde{f}(x\mid y):=f^1(x)-f^2(y)-\langle \nabla f^2(y), x-y \rangle,\\
		\tilde{c_i}(x\mid y):= \,c_i^1(x)-\,c_i^2(y)-\langle \nabla c_i^2(y)\,, x-y \rangle.
	\end{cases}$$ 
\end{proposition}
To simplify the subproblem, one may also linearize the $f^1(x)$ and $c^1(x)$ terms while adding proximal penalties suggested by the H\"olderian continuity properties.

\subsection{Nested Composite Functions}
Another important case in nonconvex optimization is the nested composite function. In this case, the objective and constraint functions possess the nested composite structure that  $f(x)=f^1(f^2(x))$ and $c_i(x)=c_i^1(c_i^2(x))$. Under the following conditions, surrogate functions that satisfy Assumption \ref{assumption:surrogate} can be easily constructed, as stated in the following proposition. As it is a direct consequence of the construction, we omit a separate proof.  

\begin{proposition}
	\label{proposition:Composite-Assp1} 
	Suppose $f^1$ and $c_i^1, i\in[m]$ are smooth convex functions. $\nabla f^1$ is $L$-Lipschitz continuous, $\nabla f^2$ is $(\ell,\kappa)$-H\"olderian continuous,  $\nabla c_i^1$ is $L_i$-Lipschitz continuous, and $\nabla c_i^2$ is $(\ell_i,\kappa_i)$-H\"olderian continuous,  $i\in[m]$.   If in addition, $\|\nabla f^1(z)\|\leq M $ and $\|\nabla c_i^1(z)\|\leq M_i$, for some finite $M,M_i>0$, $i\in[m]$. Then the convex surrogate functions 
	$$\begin{cases}
		\tilde{f}(x\mid y) \,:= f^1\big(f^2(y)+\langle \nabla f^2(y), x-y \rangle\big) + \frac{ML}{\kappa+1}\|x-y\|^{\kappa+1}, \\
		\tilde{c_i}(x\mid y):=c_i^1\big(c_i^2(y)+\langle \nabla c_i^2(y), x-y \rangle\big) \!+ \frac{M_iL_i}{\kappa_i+1}\|x-y\|^{\kappa_i+1}, 
	\end{cases}$$ 
	satisfy Assumption \ref{assumption:surrogate} within arbitrary compact region.
\end{proposition}

\section{Convex and Gradient H\"{o}lderian Continuous Case} \label{holderian_case}

In this section, we study the convex case of the main formulation $(P)$, where $f$, $r$ and $c_i, i\in[m]$ are all smooth convex functions, and $\XX\subseteq \RR^n$ is a closed convex set. To emphasize the convexity of the current problem setting, we will call $(P)$ as $(CP)$ in this section. Suppose Assumption \ref{assumption:Holder-Grad} and Assumption \ref{assumption:Slater} hold. That is, we require $f$ and $c_i,i\in[m]$ to satisfy the gradient H\"olderian continuity condition, and we require $(CP)$ to satisfy Slater's condition and bounded level set condition. According to our discussion in Section \ref{subsection:Holder-NCVX}, we can adopt the convex upper surrogate function constructed by \eqref{eqn:Holder-surrogate} and obtain the following variant of CMMA. \vspace{0.3cm}

\shadowbox{\begin{minipage}{5.2in}
		
		{\bf Algorithm GHMA (Gradient H\"{o}lderian Majorization Algorithm)}
		
		\begin{description}
			\item[Step 0:] Choose $x^0\in \XX$ to satisfy $c_i(x^0) < 0$, $i=1,2,...,m$, and let $k:=0$.
			\item[Step 1:] Solve the following subproblem with $x=x^k$:
			\[
			\begin{array}{lll}
				\!\!\!\!(CP_x)\quad & z^*(x) := \argmin_{z\in\XX} & f(x) + \nabla f(x)^\top (z-x) + \frac{L}{1+\kappa} \|z-x\|^{1+\kappa} + r(z)\\
				& \qquad\qquad\,\,\,\,\mbox{s.t.} & c_i(x) + \nabla c_i(x)^\top (z-x) + \frac{L_i}{1+\kappa_i} \|z-x\|^{1+\kappa_i} \le 0
			\end{array}
			\]
			and update $x^{k+1} = z^*(x^k)$.
			\item[Step 2:] If $x^{k+1}=x^k$ then stop. Otherwise, let $k:=k+1$, and go to {\bf Step 1}.
		\end{description}
\end{minipage}} \vspace{0.4cm}

To simplify the notation of the analysis, let us consider the special case where $\kappa_1=\kappa_2=\cdots=\kappa_m = \kappa$. We should notice that even the $\kappa_i$'s have different values, the following analysis still holds.
\begin{assumption}
	\label{assumption:convexity}
	The functions $f$ and $c_i,i\in[m]$ are all convex functions. 
\end{assumption}
Next, we start the analysis of GHMA with the following descent result. 

\begin{lemma}
	\label{lemma:descent-cvx}
	Suppose Assumptions \ref{assumption:Holder-Grad} and \ref{assumption:convexity} hold. For an arbitrary point $x\in{\rm Lev}(x^0)$, the optimal solution $z^*(x)$ to the subproblem  $(CP_{x})$ satisfies that
	\begin{equation}  
		F(x) - F(z^*(x)) \ge  \frac{\kappa}{\kappa+1} \Big(L + \sum_{i=1}^m \lambda_i L_i\Big)\|x-z^*(x)\|^{\kappa+1},
	\end{equation}
	where $\lambda\in\RR^m_+$ is the Lagrangian multiplier associated with $(CP_{x})$.
\end{lemma}
\proof Denote $d = z^*(x)-x$. Then the KKT condition of $(CP_x)$ can be written as 
\begin{equation}
	\label{eqn:GHMA-sub-KKT}
	\begin{cases}
		\Big\langle \nabla f(x) + L \|d\|^{\kappa-1}d +\sum_{i=1}^m \lambda_i \Big( \nabla c_i (x) + L_i\|d\|^{\kappa-1}d \Big) +l, z-z^*(x) \Big\rangle
		\geq 0,\,\,\, \forall z \in \XX,\\
		\lambda_i\cdot\Big(c_i(x) + \nabla c_i(x)^\top d + \frac{L_i}{\kappa+1}\|d\|^{\kappa+1}\Big) = 0, \,\, \forall i\in[m],
	\end{cases} 
\end{equation}
where $l\in\partial r(z^*(x))$. Set $z = x$ in the first relationship of \eqref{eqn:GHMA-sub-KKT} gives 
\begin{eqnarray}
	\label{lm:descent-cvx-1}
	-\nabla f(x)^\top d - l^\top d - \sum_i\lambda_i\nabla c_i(x)^\top d \geq \Big(L+\sum_i\lambda_iL_i\Big)\|d\|^{\kappa+1} . 
\end{eqnarray}
Note that the $(L,\kappa)$-H\"olderian continuity of $f$ indicates that 
\begin{equation}
	\label{lm:descent-cvx-2}
	f(x) + \nabla f(x)^\top d + \frac{L}{\kappa+1}\|d\|^{\kappa+1} - f(z^*(x)) \geq 0 . 
\end{equation}
The convexity of $r$ gives $r(x)-r(z^*(x)) + l^\top d\geq 0$. The complementary slackness of \eqref{eqn:GHMA-sub-KKT} implies
$$\lambda_i\nabla c_i(x)^\top d = -\frac{\lambda_iL_i}{\kappa+1}\|d\|^{\kappa+1} - \lambda_ic_i(x) \overset{(i)}{\geq} -\frac{\lambda_iL_i}{\kappa+1}\|d\|^{\kappa+1}$$
where (i) is because $\lambda_i\geq0$ and $c_i(x)\leq0$. Substituting them to \eqref{lm:descent-cvx-1} proves the lemma.  \hfill $\Box$

Therefore, the iterates of GHMA have a monotonically decreasing sequence in objective value,  
while the majorization scheme of the constraints retains feasibility. According to Assumption \ref{assumption:Slater}, we assume the level set ${\rm Lev}(x^0)$ is bounded. Therefore, we can define a constant $C_1>0$ as the diameter of the following set  
\begin{equation}
	\label{defn:C1}
	{\rm diam}\big({\rm Lev}(x^0)\cup\{\hat{x}\}\big)=:C_1<+\infty,
\end{equation} where $\hat{x}$ is the strictly feasible solution stated in Assumption \ref{assumption:Slater}. 
Under the convexity of $(CP)$, Assumption \ref{assum:upper} becomes a natural consequence of Assumption \ref{assumption:Slater}. We state this property as the following theorem.
\begin{lemma}
	\label{dualboundedness}
	Suppose Assumptions \ref{assumption:Slater} and \ref{assumption:convexity} hold, then there exist constants $\varrho,B_0 > 0$ such that for any feasible $x\in{\rm Lev}(x^0)$, the Lagrangian multiplier $\lambda\in\RR^m_+$ associated with the subproblem $(CP_x)$ satisfies $\|\lambda\|_\infty\leq B_0$ as long as $\|z^*(x)-x\|\leq \varrho$.
\end{lemma}

\proof 
By the second relationship of \eqref{eqn:GHMA-sub-KKT}, for any $i$ such that $\lambda_i>0$, we must have 
\[
c_i(x) + \nabla c_i(x)^\top d + \frac{L_i}{\kappa+1}\|d\|^{\kappa+1} = 0,
\]
which, together with the convexity of $c_i$,  implies that 
\begin{equation} \label{bd-1}
	\frac{L_i}{\kappa+1}\|d\|^{\kappa+1}
	\geq - c_i(z^*(x)) \ge 0.    
\end{equation}
In addition, by Lemma \ref{lemma:descent-cvx}, we know that $z^*(x)\in{\rm Lev}(x^0)$ as long as $x\in{\rm Lev}(x^0)$. Therefore,  $\|x-z^*(x)\|\leq C_1$, $\|x-\hat{x}\|\leq C_1$, and $\|z^*(x)-\hat{x}\|\leq C_1$, where the constant $C_1$ is defined in \eqref{defn:C1}. Let us denote $\hat{d} = \hat{x}-z^*(x)$. Then, we have the following bound: 
\begin{eqnarray} 
	\big\langle\nabla c_i (x) + L_i\|d\|^{\kappa-1} d,\hat{d}\,\big\rangle
	&=&  \nabla c_i (x)^\top  (\hat{x} - x) 
	- \nabla c_i(x)^\top d 
	+L_i\|d\|^{\kappa-1} d^\top \hat{d}  \nonumber\\
	& \overset{(i)}{\leq} & c_i(\hat{x}) - c_i(x) - \nabla c_i(x)^\top d + L_{i}C_1\|d\|^{\kappa}\nonumber \\
	&\overset{(ii)}{\leq} & c_i(\hat{x}) - c_{i}(z^*(x)) + \frac{L_i}{\kappa+1} \|d\|^{\kappa+1}+L_{i}C_1\|d\|^{\kappa} \nonumber\\
	& \overset{(iii)}{\leq} & c_i(\hat{x}) + 3L_iC_1\|d\|^{\kappa}.\nonumber
\end{eqnarray}
In the above arguments, (i) is due to the convexity of $c_i$, (ii) is due to inequality \eqref{eqn:descent-lemma-Holder}, and (iii) is because of \eqref{bd-1}. 
Now suppose we select $\varrho$ to be small enough so that 
\begin{equation}
	\label{eqn:cvx-varrho-1}
	\varrho \leq \min_{i\in[m]} \Big(\frac{-c_i(\hat{x})}{6L_iC_1}\Big)^{1/\kappa}.
\end{equation}
Then, as long as $\|d\| = \|z^*(x)-x\|\leq \varrho$, we have 
\begin{eqnarray}
	\label{eqn:dualbound-1}
	\big\langle\nabla c_i (x) + L_i\|d\|^{\kappa-1} d,\hat{d}\,\big\rangle
	\leq  c_i(\hat{x})/2 <0. 
\end{eqnarray} 
Similarly, using the convexity of $f$ and $r$, we obtain
\begin{eqnarray}
	\label{eqn:dualbound-2}
	\big\langle\nabla f (x) + l + L\|d\|^{\kappa-1}d,\hat{d}\,\big\rangle
	&=& \nabla f (x)^\top  (\hat x-x) -
	\nabla f (x)^\top d + l^\top\hat{d} + L\|d\|^{\kappa-1}d^\top\hat{d}\nonumber \\
	&\leq& f(\hat x) - f(x)  - 
	\nabla f (x)^\top  d
	+ r(\hat{x}) - r(z^*(x)) + LC_1^{\kappa+1} \\
	&\le& f(\hat x) - f(z^*(x)) +\frac{L}{\kappa+1}\|d\|^{\kappa+1} + r(\hat{x}) - r(z^*(x)) + LC_1^{\kappa+1}\nonumber \\
	&\le& F(\hat{x}) - F(z^*(x)) + 2LC_1^{\kappa+1}.\nonumber
\end{eqnarray}
Now substituting $z = \hat{x}$ into \eqref{eqn:GHMA-sub-KKT} and using the inequalities \eqref{eqn:dualbound-1} and \eqref{eqn:dualbound-2}, we obtain 
$$\sum_{j: \lambda_j>0}-c_j(\hat{x})\lambda_j/2\leq F(\hat{x})-F^* + 2LC_1^{\kappa+1}.$$
Note that $-c_j(\hat{x})>0$ and $\lambda_j>0$ for all summed $j$, while other $\lambda_j=0$. We know 
\begin{equation}
	\label{eqn:cvx-B0}
	\|\lambda\|_\infty\leq \frac{2(F(\hat{x})-F^*)+4LC_1^{\kappa+1}}{\min\{-c_1(\hat{x}),\cdots,-c_m(\hat{x})\}}=:B_0.
\end{equation}
That is, Assumption \ref{assum:upper} holds with small enough $\varrho$ that satisfies \eqref{eqn:cvx-varrho-1} and $B_0$ given by \eqref{eqn:cvx-B0}.
\hfill $\Box$ 

As a direct corollary of Lemma \ref{lemma:descent-cvx} and Lemma \ref{dualboundedness}, we have the following result.
\begin{corollary}
	\label{corollary:LagrangianBound}
	There is a constant $B>0$, so that for every iteration $k\geq 0$, the Lagrangian multiplier $\lambda^{k+1}\in\RR^m_+$ associated with $(CP_{x^k})$ satisfies $0 \le \|\lambda^{k+1}\|_\infty \leq B$. 
\end{corollary} 
The proof is straightforward in that the sufficient descent in Lemma \ref{lemma:descent-cvx} guarantees  $\|d^k\|\geq\varrho$ to happen at most finitely. Therefore, one can set $B$ as the maximum between  $B_0$ and the finitely many $\|\lambda^{k+1}\|_\infty$'s that correspond to $\|d^k\|\geq \varrho$. The proof is omitted for succinctness.

For convex optimization $(CP)$, beyond the iterative sufficient descent in Lemma \ref{lemma:descent-cvx}, we further provide the following bound on the function value optimality gap, which, together with Corollary~\ref{corollary:LagrangianBound}, provides the final complexity result. 

\begin{lemma} \label{lemma:opt-gap}
	Under the same setting of Lemma \ref{lemma:descent-cvx}, and let $x^*$ be the optimal solution of $(CP)$, then
	\begin{equation}
		\label{lm:opt-gap-main-1} 
		F(z^*(x)) - F(x^*) \leq \frac{L+\!\sum_{i=1}^m\lambda_iL_i}{2}\!\cdot\!\|d\|^{\kappa-1}\!\cdot\! \Big(\frac{1-\kappa}{1+\kappa}\|d\|^2 + \|x-x^*\|^2 - \|z^*(x)-x^*\|^2\Big).
	\end{equation}
	Meanwhile, with upper bound $C_1$ as given in \eqref{defn:C1}, we also have
	\begin{equation}
		\label{lm:opt-gap-main-2}
		F(z^*(x))-F(x^*) \leq \Big(L + \sum_{i=1}^m \lambda_i L_i  \Big)C_1\|d\|^{\kappa}, 
	\end{equation} 
	where $d = z^*(x)-x$ denotes the algorithmic residual of the subproblem $(CP_x)$.
\end{lemma}
\proof 
First, substituting and $z = x^*$ in the first equality of \eqref{eqn:GHMA-sub-KKT} and rearranging the terms gives 
\begin{equation} 
	\label{lm:opt-gap-1}
	\langle \nabla f(x) + l,x^*-z^*(x)\rangle + \sum_i\lambda_i\nabla c_i(x)^\top(x^*-z^*(x)) + \Big(L+\sum_{i=1}^m\lambda_iL_i\Big)\|d\|^{\kappa-1}d^\top(x^*-z^*(x)) \geq 0, 
\end{equation}
where $d = z^*(x)-x$ and $l\in\partial r(z^*(x))$. Using the convexity of $r$ and $f$, we obtain
\begin{equation} 
	\label{lm:opt-gap-2}
	r(x^*) \geq r(z^*(x)) + l^\top(x^*-z^*(x))\qquad\mbox{and}\qquad f(x^*) \geq f(x) + \nabla f(x)^\top(x^*-x). 
\end{equation}
Summing up inequalities \eqref{lm:opt-gap-1}, \eqref{lm:opt-gap-2}, and \eqref{lm:descent-cvx-2}, and then rearranging the terms yields 
\begin{equation} 
	\label{lm:opt-gap-3}
	F(z^*(x)) - F(x^*) \leq \underbrace{\sum_{i=1}^m\lambda_i\nabla c_i(x)^\top(x^*-z^*(x))}_{T_1} + \underbrace{\Big(L+\!\sum_{i=1}^m\lambda_iL_i\Big)\|d\|^{\kappa-1}d^\top(x^*-z^*(x))}_{T_2} + \frac{L\|d\|^{\kappa+1}}{\kappa+1}. 
\end{equation} 
To bound the term $T_1$, let us add the complementary slackness terms in \eqref{eqn:GHMA-sub-KKT} to $T_1$, then we have 
\begin{eqnarray*}
	T_1 & = & T_1 + \sum_{i=1}^m\lambda_i\cdot\Big(c_i(x) + \nabla c_i(x)^\top d + \frac{L_i}{\kappa+1}\|d\|^{\kappa+1}\Big)\\
	& = & \sum_{i=1}^m\lambda_i\cdot\Big(c_i(x) + \nabla c_i(x)^\top (x^*-x) + \frac{L_i}{\kappa+1}\|d\|^{\kappa+1}\Big)\\
	& \leq & \sum_{i=1}^m \frac{\lambda_iL_i}{\kappa+1}\|d\|^{\kappa+1}  
\end{eqnarray*}
where the last inequality is due to the convexity of $c_i$, the feasibility of $x^*$ and $\lambda\geq0$, which implies 
$$\sum_{i=1}^m\lambda_i\Big(c_i(x) + \nabla c_i(x)^\top (x^*-x)\Big) \leq \sum_{i=1}^m\lambda_ic_i(x^*)\leq 0.$$
To bound $T_2$, it suffices to use the identity  
$$(a-b)^\top(c-d) = \frac{1}{2}\Big(\|a-d\|^2-\|a-c\|^2\Big) + \frac{1}{2}\Big(\|b-c\|^2-\|b-d\|^2\Big)$$
that holds for any vectors $a,b,c,d$, which indicates that  
$$T_2 = \frac{1}{2}\Big(L+\sum_{i=1}^m\lambda_iL_i\Big)\cdot\|d\|^{\kappa-1}\Big(\|x-x^*\|^2 - \|z^*(x)-x^*\|^2 - \|d\|^2\Big).$$
Substituting the bounds on $T_1$ and $T_2$ gives 
\begin{eqnarray*}
	&&F(z^*(x)) - F(x^*) \\
	&\leq &\frac{L+\!\sum_{i=1}^m\lambda_iL_i}{2}\!\cdot\!\|d\|^{\kappa-1}\!\cdot\!\underbrace{\Big(\frac{1-\kappa}{1+\kappa}\|x - z^*(x)\|^2 + \|x-x^*\|^2 - \|z^*(x)-x^*\|^2\Big)}_{T_3}\nonumber
\end{eqnarray*}
where the term $T_3$ can be further upper bounded by 
\begin{eqnarray} 
	T_3 & = & (x  - z^*(x))^\top \Big(x + z^*(x) - 2x^*+\frac{1-\kappa}{1+\kappa}\big(x  - z^*(x)\big)\Big)\nonumber \\
	& = & -d^\top \Big(\frac{2}{1+\kappa}(x-x^*)+\frac{2\kappa}{1+\kappa}\big(z^*(x)-x^*\big)\Big)\nonumber\\
	& \leq & \|d\| \Big(\frac{2}{1+\kappa}\|x -x^*\|+\frac{2\kappa}{1+\kappa}\|z^*(x) -x^*\|\Big)\nonumber\\
	& \le & 2 C_1 \|d\|.\nonumber
\end{eqnarray}
Substituting the bound on $T_3$ to \eqref{lm:opt-gap-3} proves the lemma.  \hfill $\Box$

Combining Lemma \ref{lemma:descent-cvx}, Lemma \ref{lemma:opt-gap}, and Corollary \ref{corollary:LagrangianBound}, we arrive at the final theorem. 

\begin{theorem}
	\label{theorem:Cvg-CVX}
	Suppose Assumptions \ref{assumption:Slater}, \ref{assumption:Holder-Grad} and \ref{assumption:convexity} hold, and the iteration sequence $\{x^k\}_{k\geq0}$ is generated by GHMA. Then $\{x^k\}_{k\geq0}$ are all feasible to $(CP)$, and  for any $K\geq 1$, we have 
	$$F(x^K)-F(x^*)\leq \frac{C_1^{1-\kappa}L_B^{\frac{1-\kappa}{\kappa+1}}}{2^\kappa}\cdot\Big(\frac{F(x^0)-F(x^*)}{\kappa K}\Big)^{\frac{2\kappa}{\kappa+1}} + \frac{C_1^{1-\kappa}L_B\|x^0-x^*\|^{2\kappa}}{2^\kappa K^\kappa},$$
	where $L_B = L+\sum_{i=1}^mBL_i$. Therefore, it takes GHMA at most $O\big(\frac{L_B^{1/\kappa}\|x^0-x^*\|^2}{\epsilon^{1/\kappa}}\big)$ iterations to obtain a feasible solution with $\epsilon$-small optimality gap. 
\end{theorem}
\proof First, by Lemma \ref{lemma:descent-cvx}, we know $F(x^k)$ is monotonically decreasing. Then let us sum up the inequalities \eqref{lm:opt-gap-main-1} from Lemma \ref{lemma:opt-gap} with $x = x^k, z^*(x) = x^{k+1}$ and $\lambda = \lambda^{k+1}$. To simplify the notation, let us denote the function gap as $\Delta_k = F(x^k)-F(x^*)$,  the weighted Lipschitz constants $L_\lambda^k:=L+\sum_{i=1}^m\lambda_i^{k+1}L_i$ and $L_B:=L+\sum_{i=1}^mBL_i$, the algorithmic residual $d^k = x^{k+1}-x^k$, then we obtain the following inequality on the optimality gap
\begin{eqnarray}
	\label{thm:Cvg-CVX-1}
	\Delta_K &\leq&  \frac{1}{K}\sum_{k=0}^{K-1}\Delta_{k+1} \nonumber\\
	& \leq &  \frac{1}{K}\sum_{k=0}^{K-1} \frac{L_\lambda^k}{2}\!\cdot\!\|d^k\|^{\kappa-1}\!\cdot\!\Big(\frac{1-\kappa}{1+\kappa}\|d^k\|^2 + \|x^k-x^*\|^2 - \|x^{k+1}-x^*\|^2\Big) \\
	& = & \underbrace{\frac{1-\kappa}{2K} \sum_{k=0}^{K-1} \frac{L_\lambda^k\|d^k\|^{\kappa+1}}{1+\kappa}}_{E_1} + \underbrace{\frac{1}{2K}\sum_{k=0}^{K-1} L_\lambda^k\|d^k\|^{\kappa-1}\cdot(\|x^k-x^*\|^2 - \|x^{k+1}-x^*\|^2)}_{E_2}\nonumber.
\end{eqnarray}
By Lemma \ref{lemma:descent-cvx}, summing up its descent result with $x = x^k, z^*(x) = x^{k+1}$ and $\lambda = \lambda^{k+1}$ gives 
$$E_1\leq \frac{1-\kappa}{2\kappa K}\big(F(x^0)-F(x^K)\big)\leq \frac{1-\kappa}{2\kappa K}\Delta_0.$$
For the term $E_2$,  the following inequality holds 
\begin{eqnarray*}
	E_2 & \overset{(i)}{\leq} & \frac{1}{2K} \sum_{k=0}^{K-1} L_\lambda^k\Big(\frac{L_\lambda^kC_1}{\Delta_{k+1}}\Big)^{ 1/\kappa-1}\cdot(\|x^k-x^*\|^2 - \|x^{k+1}-x^*\|^2)\\
	& \overset{(ii)}{\leq} & \frac{1}{2K} \sum_{k=0}^{K-1} L_B^{{1}/{\kappa}}(C_1/\Delta_K)^{{1}/{\kappa}-1}\big(\|x^k-x^*\|^2 - \|x^{k+1}-x^*\|^2\big)\\
	& = & \frac{L_B^{{1}/{\kappa}}(C_1/\Delta_K)^{{1}/{\kappa}-1}}{2K} (\|x^0-x^*\|^2 - \|x^{K}-x^*\|^2)\\
	& {\leq} & \frac{L_B^{{1}/{\kappa}}(C_1)^{{1}/{\kappa}-1}}{2K\Delta_K^{{1}/{\kappa}-1}}\cdot\|x^0-x^*\|^2.
\end{eqnarray*}
In the above argument, (i) is because the inequality \eqref{lm:opt-gap-main-2} of Lemma \ref{lemma:opt-gap} indicates that $\|d^k\|\geq \big(\frac{\Delta_{k+1}}{L_\lambda^kC_1}\big)^{1/\kappa}$, together with the fact that $\kappa\in(0,1]$ and $\kappa-1\leq 0$, we can use $\Delta_{k+1}$ to provide an upper bound on $\|d^k\|^{\kappa-1}$. (ii) is because Corollary \ref{corollary:LagrangianBound} indicates that $\|\lambda^k\|_\infty \leq B$ for any $k\geq0$, which further implies that $L_\lambda^k\leq L_B$ for all $k\geq 0.$ We also use the monotonicity of $F(x^k)$ which indicates that $\Delta_k\geq \Delta_K$ for all $k\leq K$. Overall, substituting the bounds on $E_1$ and $E_2$ to \eqref{thm:Cvg-CVX-1} and then multiply $\Delta_K^{{1}/{\kappa}-1}$ on both sides, we obtain 
\begin{equation}
	\label{thm:Cvg-CVX-2}
	\Delta_K^{{1}/{\kappa}}\leq \frac{1-\kappa}{2\kappa K}\Delta_0\Delta_K^{{1}/{\kappa}-1} + \frac{L_B^{{1}/{\kappa}}(C_1)^{{1}/{\kappa}-1}}{2K}\cdot\|x^0-x^*\|^2.
\end{equation} 
Let us set $s := \argmin_{k} \left\{L_\lambda^k\|d^k\|^{\kappa+1} : 0\leq k\leq K-1\right\}$, then summing up the descent inequality of Lemma \ref{lemma:descent-cvx} with $x = x^k, z^*(x) = x^{k+1}$ and $\lambda = \lambda^{k+1}$ gives
$$L_\lambda^{s}\|d^s\|^{\kappa+1}\leq \frac{1}{K}\sum_{k=0}^{K-1}L_\lambda^{k}\|d^k\|^{\kappa+1} \leq \frac{\kappa+1}{\kappa K}\Delta_0.$$
Again, by the inequality \eqref{lm:opt-gap-main-2} of Lemma \ref{lemma:opt-gap}, we have
$$\Delta_K\leq \Delta_{s+1} \leq L_\lambda^sC_1\|d^s\|^\kappa \leq \Big(\frac{(\kappa+1)\Delta_0}{\kappa K}\Big)^{\frac{\kappa}{\kappa+1}}(L_B)^{\frac{1}{\kappa+1}}C_1.$$
Substituting this bound to the right hand side of \eqref{thm:Cvg-CVX-2} gives 
$$\Delta_K^{{1}/{\kappa}}\leq \frac{L_B^{\frac{1-\kappa}{\kappa(\kappa+1)}}C_1^{1/\kappa-1}}{2}\cdot\Big(\frac{\Delta_0}{\kappa K}\Big)^{\frac{2}{\kappa+1}} + \frac{L_B^{{1}/{\kappa}}C_1^{{1}/{\kappa}-1}}{2K}\cdot\|x^0-x^*\|^2.$$
Note that for any $a,b>0$ and $\kappa\in(0,1]$, it holds that $(a+b)^\kappa\leq a^\kappa+b^\kappa$. Hence we obtain
$$\Delta_K\leq \frac{L_B^{\frac{1-\kappa}{\kappa+1}}C_1^{1-\kappa}}{2^\kappa}\cdot\Big(\frac{\Delta_0}{\kappa K}\Big)^{\frac{2\kappa}{\kappa+1}} + \frac{L_BC_1^{1-\kappa}}{2^\kappa K^\kappa}\cdot\|x^0-x^*\|^{2\kappa} \leq O\Big(\frac{1}{K^\kappa}\Big).$$ 
This completes the proof. \hfill $\Box$

Under the general convexity condition and gradient H\"olderian continuity, an $O\big({1}/{K^\kappa}\big)$ sublinear convergence can be achieved by GHMA. If in addition, $\kappa=\kappa_1=\cdots=\kappa_m = 1$, and a strong convexity condition is satisfied, then a fast linear convergence can be achieved. 

\begin{theorem} \label{theorem:Cvg-SCVX}
	Suppose Assumptions \ref{assumption:Slater}, \ref{assumption:Holder-Grad} and \ref{assumption:convexity} hold. In addition, assume that $\kappa=1$ and $f$ is $\alpha$-strongly convex, then the iteration sequence $\{x^k\}_{k\geq0}$ generated by GHMA satisfies
	$$\frac{\alpha}{2}\|x^{k}-x^*\|^2  + \big(1-\rho^k\big)\big(F(x^k)-F(x^*)\big) \leq \frac{\alpha\rho^k}{2}\|x^0-x^*\|^2, \quad \forall k\geq0,$$
	where $L_B = L+B\sum_iL_i$ and $\rho = 1-\alpha/L_B$. Therefore, it takes GHMA at most $O\big(\frac{L_B}{\alpha}\log\big(\frac{\alpha\|x^0-x^*\|^2}{\epsilon}\big)\big)$ iterations to obtain a feasible solution with $\epsilon$-small optimality gap. 
\end{theorem} 
\proof By the $\alpha$-strong convexity of $f$, let us replace \eqref{lm:opt-gap-2} with 
$$f(x^*)\geq \nabla f(x^k) + \nabla f(x^k)^\top(x^*-x^k) + \frac{\alpha}{2}\|x^k-x^*\|^2$$
and let us set $x = x^k, z^*(x) = x^{k+1},\lambda = \lambda^{k+1}$ and $\kappa=1$ in the proof of \eqref{lm:opt-gap-main-1}. Then following the same line of proof in Lemma \ref{lemma:opt-gap} with \eqref{lm:opt-gap-2} replaced by the above strongly convex version, we obtain the following strongly convex counterpart of \eqref{lm:opt-gap-main-1}:
\begin{eqnarray}
	\label{thm:cvg-scvx-1}
	F(x^{k+1}) - F(x^*) \leq \frac{L_\lambda^k-\alpha}{2}\|x^k-x^*\|^2 - \frac{L_\lambda^k}{2}\|x^{k+1}-x^*\|^2.\nonumber
\end{eqnarray}
Denote $\Delta_{k} := F(x^k)-F(x^*)$ and $\rho = 1-\alpha/L_B$. The above inequality implies that  
$$\frac{2\Delta_{k+1}}{L_B} \leq \frac{2\Delta_{k+1}}{L_\lambda^k} \leq \rho\|x^{k}-x^*\|^2 - \|x^{k+1}-x^*\|^2.$$
For any $K\geq1$, let us take the weighted summation of the above inequality over $k = 0,1,\cdots,K-1$ with weight $\rho^{K-k-1}$, we obtain the following bound
\begin{eqnarray*}
	\rho^K\|x^{0}-x^*\|^2 &\geq& \|x^{K}-x^*\|^2 + \frac{2}{L_B}\sum_{k=0}^{K-1}\rho^k\Delta_{K-k+1}\\
	&\geq & \|x^{K}-x^*\|^2 + \frac{2}{L_B}\sum_{k=0}^{K-1}\rho^k\Delta_{K}\\
	& = & \|x^{K}-x^*\|^2  + \frac{2}{\alpha}\Big(1-\rho^K\Big)\Delta_K.
\end{eqnarray*} 

To guarantee $\Delta_K \le \epsilon$, it suffices to have $\rho^K = O\big(\frac{2 \epsilon}{\alpha \|x^0 - x^*\|^2}\big).$  Taking logarithms on both sides and using $\rho=1- \alpha / L_B$, we obtain
\[
K = O\Big(\frac{\log(\alpha \|x^0 - x^*\|^2 / 2\epsilon)}{-\log \rho}\Big) = O\Big(\frac{L_B}{\alpha}\log\big(\frac{\alpha\|x^0-x^*\|^2}{\epsilon}\big)\Big),
\]
which completes the proof of the theorem.  
\hfill $\Box$

\section{Solving the Subproblem}
\label{Solving_Subproblem}

In this section, we will discuss how to solve the subproblem efficiently arising from the Lipschitz/H\"older continuous cases. In general, we will solve the problems of the form
\begin{eqnarray}
	\label{prob:subproblem}
	& \min_{p\in\RR^n} & a_0^\top p + \frac{L}{\kappa+1}\cdot\|p\|^{\kappa+1} \\
	& \mathrm{s.t.} & a_i^\top p + \frac{L_i}{\kappa_i+1}\cdot\|p\|^{\kappa_i+1} \leq b_i, \,\, \mbox{ for }\,\, i\in [m].\nonumber
\end{eqnarray}
Note that this is a strictly convex optimization problem given $\kappa_i\in(0,1], i=0,1,\cdots,m$. When $\kappa=1$, the problem becomes strongly convex. To solve this problem, let us rewrite the problem as 
$$\min_{p\in\RR^n} \max_{\lambda\geq0} \,\,\mathcal{L}(p,\lambda):=\left(a_0 + \sum_{i=1}^m\lambda_ia_i\right)^\top p + \frac{L}{\kappa+1}\cdot\|p\|^{\kappa+1}+\sum_{i=1}^m\frac{\lambda_iL_i}{\kappa_i+1}\cdot\|p\|^{\kappa_i+1} - b^\top\lambda.$$
Denote $A^\top = [a_1, a_2,\cdots,a_m]\in\RR^{n\times m}$. Then it is not hard to observe that $\mathcal{L}(p,\lambda)$ has a unique $p^*_\lambda$ that satisfies 
\begin{equation}
	\label{eqn:KKT-subproblem}
	a_0 + A^\top\lambda + L\|p^*_\lambda\|^{\kappa-1}\cdot p_\lambda^*+\sum_{i=1}^m \lambda_iL_i\|p^*_\lambda\|^{\kappa_i-1}\cdot p_\lambda^* = 0,
\end{equation} 
which further indicates that 
$\|a_0 + A^\top\lambda\| = \psi(\|p_\lambda^*\|)$, where $\psi(\alpha) = L \alpha^{\kappa} +\sum_{i=1}^m \lambda_iL_i\alpha^{\kappa_i}$ is a strictly monotonically increasing function on $[0,+\infty)$ as $L>0$, $\lambda_iL_i\geq0$, and $\kappa_i\in(0,1], 0\leq i\leq m$. Let $\alpha_\lambda^*\geq0$ be the unique non-negative real solution of the equation $\psi(\alpha) = \|a_0 + A^\top\lambda\|$, then we have 
$$p_\lambda^* = -\alpha_\lambda^*\cdot\frac{a_0+A^\top\lambda}{\|a_0+A^\top\lambda\|}.$$
In general, the solution $\alpha_\lambda^*$ can be easily obtained via binary search or secant method. When $\kappa = \cdots = \kappa_m = \kappa$, then a closed form solution will be available:
$$\alpha_\lambda^* = \left(\frac{\|a_0+A^\top\lambda\|}{L+\sum_{i=1}^m\lambda_iL_i}\right)^{1/\kappa}.$$
When $\kappa=1$, then this corresponds to the Lipschitz continuous gradient case, \textcolor{black}{which has been discussed in \cite{auslender2010moving}}. For notational simplicity, let us default $\lambda_0 = 1$, which is not included in the dual decision variable $\lambda$. Given the unique solution, we can obtain the dual problem of the subproblem \eqref{prob:subproblem}:
$$\max_{\lambda\geq0}\,\, \Phi(\lambda):=  - \sum_{i=0}^m\frac{\kappa_i\lambda_iL_i}{\kappa_i+1}\|p_\lambda^*\|^{\kappa_i+1} - b^\top\lambda$$
where $\Phi$ is derived by substituting \eqref{eqn:KKT-subproblem} to $\mathcal{L}(p_\lambda^*,\lambda)$. Note that by Danskin's theorem, we can compute the gradient of $\Phi$ as
$$\frac{\partial \Phi(\lambda)}{\partial \lambda_i} = \frac{\partial \mathcal{L}(p_\lambda^*,\lambda)}{\partial \lambda_i}  = a_i^\top p_\lambda^* + \frac{L_i}{\kappa_i+1}\|p_\lambda^*\|^{\kappa_i+1}-b_i,\quad\mbox{for}\quad i \in [m],$$
which corresponds to the primal constraint violation if $\frac{\partial \Phi(\lambda)}{\partial \lambda_i}>0$.  By the implicit function theorem, we can compute the Hessian of $\Phi$ as
\begin{eqnarray}
	\label{eqn:Hess-subprob}
	\nabla^2\Phi(\lambda) & = & \nabla^2_{\lambda\lambda}\mathcal{L}(p_\lambda^*,\lambda) - \nabla^2_{\lambda p}\mathcal{L}(p_\lambda^*,\lambda)\big[\nabla^2_{pp}\mathcal{L}(p_\lambda^*,\lambda)\big]^{-1}\nabla^2_{p\lambda}\mathcal{L}(p_\lambda^*,\lambda).
\end{eqnarray} 
Let $\tilde{L}\in\RR^m$ be a vector with $\tilde{L}_i = L_i\|p_\lambda^*\|^{\kappa_i-1}$, for $i = 1,2,...,m$, then straightforward computation gives $\nabla^2_{\lambda\lambda}\mathcal{L}(p_\lambda^*,\lambda) = 0_{m\times m}$,  $\nabla^2_{p\lambda}\mathcal{L}(p_\lambda^*,\lambda)=\big[\nabla^2_{\lambda p}\mathcal{L}(p_\lambda^*,\lambda)\big]^\top = A^\top+ p_\lambda^*\tilde{L}^\top$ and 
\begin{eqnarray}
	\nabla^2_{pp}\mathcal{L}(p_\lambda^*,\lambda) & = & \left(\sum_{i=0}^m\lambda_iL_i\|p_\lambda^*\|^{\kappa_i-1}\right)\cdot I_{n\times n} + \left( \sum_{i=0}^m\lambda_iL_i(\kappa_i-1)\|p_\lambda^*\|^{\kappa_i-3}\right)\cdot p_\lambda^*(p_\lambda^*)^\top.\nonumber 
\end{eqnarray} 
Consequently, noting $\kappa_i\in (0,1]$ for $i\in[m]$, we have 
$$\left(\sum_{i=0}^m\lambda_iL_i\|p_\lambda^*\|^{\kappa_i-1}\right)\cdot I_{n\times n}\succeq \nabla^2_{pp}\mathcal{L}(p_\lambda^*,\lambda)\succeq \left(\sum_{i=0}^m\lambda_iL_i\kappa_i\|p_\lambda^*\|^{\kappa_i-1}\right)\cdot I_{n\times n},$$
and the matrix $\nabla^2_{pp}\mathcal{L}(p_\lambda^*,\lambda)\succ0$ as long as $\lambda\neq 0$. Note that 
$$\|p_\lambda^*\| = \alpha_\lambda^* \leq \left(\frac{\|a_0+A^\top\lambda\|}{L}\right)^\frac{1}{\kappa}\leq \left(\frac{\|a_0\|+\|A\|\|\lambda\|}{L}\right)^\frac{1}{\kappa}.$$ 
By Assumption \ref{assum:upper}, there is an optimal $\|\lambda^*\|\leq B$. As the iterates' distance to the optimal solution is always non-increasing in many convex optimization algorithms such as (projected) gradient descent, the dual updates will be uniformly upper bounded by some constant $2B$ if one starts from $\lambda^0=0$. Suppose the matrix $A$ has full rank, i.e., $\sigma_0:=\sigma_{\min}(A)>0$. Then 
$$\nabla^2\Phi(\lambda)\succeq \frac{\sigma_0^2/4}{\sum_{i=0}^m\lambda_iL_i\|p_\lambda^*\|^{\kappa_i-1}}\cdot I_{m\times m}\succ 0_{m\times m}$$
if $L\geq \frac{2\|\tilde{L}\|\|a_0\|}{\sigma_0} + \left(\frac{2\|A\|}{\sigma_0}-1\right)2B\|\tilde{L}\|$. Thus, in this case the subproblem will exhibit a linear convergence if $\lambda^*$ is away from $0$. Such a linear convergence of subproblem is often observed in the numerical experiments. However, when the above requirements fail, then only an $O(1/k^2)$ sub-linear convergence rate can be observed in the subproblem if Nesterov's accelerated gradient method is applied to the dual problem. 

As a result, one can easily implement the gradient method with line search to quickly solve the dual problem. If the constraint number $m$ is mild, one may even apply the Projected Newton's method \cite{bertsekas1982projected}, which typically returns a highly accurate solution with only a few iterations. By properly utilizing the Sherman-Morrison formula, the per-iteration computational cost of Newton's method can be greatly reduced. After solving the dual problem, one can return $p_\lambda^*$ as the solution of the primal subproblem \eqref{prob:subproblem}.

\section{More Variants Based On Convex Majorization} 
\label{Barrier_Dikin}

\subsection{A Barrier Variant}

Now, assume that the objective and constraints are all gradient Lipschitz continuous, with Lipschitz constants $L$ and $L_i$ respectively, where $i=1,2,...,m$. In the case when $m$ is large, solving the subproblem of Algorithm GHMA may still be time-consuming. Instead, it is cheaper to consider a barrier-variant for the subproblem. Introduce the self-concordant logarithmic barrier of the quadratic constrained feasible region: 
\[
\tilde B_k(x) =
- \sum_{i=1}^m \ln \left( - c_i(x^k) - \nabla c_i(x^k)^\top (x-x^k) - \frac{L_i\|x-x^k\|^2}{2} \right) . 
\]
Let $\mu>0$ be a fixed algorithmic parameter. Then the Convex Envelope with Barriers (CEB) Algorithm is presented below. \vspace{0.3cm}

\shadowbox{\begin{minipage}{4.6in}
		{\bf Algorithm CEB (Convex Envelope with Barriers)}
		
		\begin{description}
			
			\item[Step 0:] Choose $c_i(x^0) < 0$, $i\in [m]$. Let $k:=0$.
			
			\item[Step 1:] Solve the following subproblem with $x=x^k$:
			\[
			z^*(x) := \mbox{arg}\min \, \nabla f(x^k)^\top (x-x^k) + L \|x-x^k\|^2 + \mu \tilde B_k(x)+r(x),
			\]
			and update $x^{k+1} = z^*(x^k)$.
			
			\item[Step 2:] If $\|x^k-x^{k+1}\| > \epsilon$ then $k:=k+1$, and return to {\bf Step 1}; else, stop. 
			
		\end{description}
\end{minipage}}
\vspace{0.15cm}

\noindent We observe that 
\begin{eqnarray}
	&&f(x^{k+1}) + r(x^{k+1}) +\mu \tilde B_k(x^{k+1})\nonumber\\
	&\leq& f(x^k)+ \nabla f(x^k)^\top (x^{k+1}-x^k)+ \frac{L
	}{2} \| x^{k+1}-x^k \|^2 +r(x^{k+1})+\mu \tilde B_k(x^{k+1})\nonumber\\
	&\leq &f(x^k) +r(x^{k})+\mu \tilde B_k(x^{k})-\frac{L
	}{2} \| x^{k+1}-x^k \|^2\nonumber,
\end{eqnarray}
where the second inequality is because $x^{k+1}$ optimizes the subproblem.
Given the fact that
\begin{eqnarray*} 
	\tilde B_k(x^{k+1}) &=&
	- \sum_{i=1}^m \ln \left( - c_i(x^k) - \nabla c_i(x^k)^\top (x^{k+1}-x^k) - \frac{L_i\|x^{k+1}-x^k\|^2}{2} \right) \\
	&\geq& - \sum_{i=1}^m \ln \left( -c_i(x^{k+1})\right)\\
	&=& \tilde B_{k+1}(x^{k+1}),
\end{eqnarray*}
we have
\[
f(x^{k+1}) +r(x^{k+1}) +\mu \tilde B_{k+1}(x^{k+1})\leq f(x^k) +r(x^{k}) +\mu \tilde B_k(x^{k})\nonumber-\frac{L
}{2} \| x^{k+1}-x^k \|^2.
\]
Denote $d^k=x^{k+1}-x^k$ and $\Delta=F(x^0)-F(x^*)$, where $x^*$ is a minimum point of the function $F
(x)=f(x)+r(x)-\mu \sum_{i=1}^m \ln \left( - c_i(x)\right)$. According to Assumption \ref{assumption:Slater}, we have $\Delta\geq 0$. Telescoping the above inequality gives 
$$\sum_{k=0}^{K-1}\frac{L}{2}\|x^{k+1}-x^k\|^{2}\leq F(x^0)-F^*=\Delta, $$
which further implies that 
\begin{equation}
	\label{eqn:dkmin0}
	\min_{0\leq k\leq K-1} \|d^k\|\leq \left(\frac{2\Delta}{LK}\right)^{1/2}.
\end{equation}
That is, the algorithmic residual $\|d^k\|\rightarrow 0$ at an $O(k^{-\frac{1}{2}})$ sublinear rate. Similar to the discussion in previous sections, we set 
${\rm Lev}(x^0):=\{x\in\XX:F(x)\leq F(x^0), c_i(x)\leq 0, i\in[m]\}$ the level set of any feasible solution $x^0$ to problem $(P)$.
According to the optimality condition, there exists an $l\in \partial r(x^{k+1})$ such that
\begin{equation}
	\nabla 
	f(x^k)+ l + L (x^{k+1}-x^k) +\mu \nabla \tilde B_k(x^{k+1})=0, \label{first-opt0}\nonumber
\end{equation}
with 
\begin{eqnarray}
	\nabla \tilde  B_k(x^{k+1})= - \sum_{i=1}^m \frac{-\nabla c_i(x^k)-L_i(x^{k+1}-x^k)} {\left( - c_i(x^k) - \nabla c_i(x^k)^\top (x^{k+1}-x^k) - \frac{L_i\|x^{k+1}-x^k\|^2}{2} \right)} . \label{barr0}\nonumber
\end{eqnarray}

So, we have the following equation
\begin{equation}
	\label{eqn0:grad-conti-kkt}
	\nabla 
	f(x^k)+ L(x^{k+1}-x^k) + \sum_{i=1}^m \lambda_i(\nabla c_i(x^k)+L_i(x^{k+1}-x^k)) =0.
\end{equation}
where $\lambda_i=\frac{-\mu}{ c_i(x^k) + \nabla c_i(x^k)^\top (x^{k+1}-x^k) + \frac{L_i\|x^{k+1}-x^k\|^2}{2}}>0$. Next, let us present an alternative
condition that guarantees the bound of $\lambda$. A similar result can be found in \cite{wei2025sqp}. As a consequence of \eqref{eqn0:grad-conti-kkt}, it remains to verify Assumption \ref{assum:upper}, that is, the Lagrangian multipliers are bounded when $\|x^k-x^{k+1}\|$ is sufficiently small. Then the next proposition shows that the conditions in Proposition \ref{proposition:Gen-LICT-to-Assp3} is sufficient for this purpose.

\begin{proposition} \label{proposition:barrier}
	Assumption \ref{assumption:Holder-Grad} with $\kappa=\kappa_i=1, i\in [m]$ together with the assumption in Proposition \ref{proposition:Gen-LICT-to-Assp3} implies that there exist $B,\varrho>0$ such that for any subproblem $(P_x)$ with $x\in{\rm Lev}(x^0)$ being strictly feasible, then $\lambda\in\RR^m_+$ satisfies $\|\lambda\|_\infty\leq B$ as long as $\|z^*(x)-x\|\leq \varrho$ for small enough $\mu$. (Note: $z^*(x)$ is defined in Step 1 of Algorithm CEB.)\\
\end{proposition}

\begin{proof}
	Let us suppose $\|d\|\leq \varrho$ 
	for some small enough $\varrho$ s.t.
	\begin{equation}
		\label{eqn:gen-LICQ-varrho2}
		\delta \,\,\geq\,\, \max_{i\in[m]} \Big\{\frac{M_i\varrho}{2} + \frac{L_i
			\varrho^{2}}{4}\Big\}.
	\end{equation}
	
	For $i\notin I(x)$, we have $c_i(x) < -\delta$. By \eqref{eqn:gen-LICQ-varrho}, definition of $\tilde{c_i}$, and condition \textbf{(ii)} in Assumption \ref{assumption:Holder-Grad}, we have
	\begin{eqnarray} 
		\big|\tilde{c_i}(z^*(x)\mid x)-c_i(x)\big| 
		\leq M_i\|d\|+\frac{L_i}{2}\|d\|^{2} \leq  \delta/2.\nonumber
	\end{eqnarray}
	This further implies that $\tilde{c_i}(z^*(x)\mid x)<-\delta/2$ and hence $\lambda_i\leq \frac{2\mu}{\delta}$. So, we have $\lambda_i\leq \frac{2\mu}{\delta}\leq  M/\delta'$ for small enough $\mu$.

	For $i\in I(x^k)$, which means $c_i(x^k)> -\delta$, we have the same inequality \eqref{ineq:lambdabound} as in the proof of Proposition \ref{proposition:Gen-LICT-to-Assp3}:
	\begin{equation} 
		\lambda_i\delta' \leq M+2L\varrho^\kappa+\sum_{j=1}^m \lambda_j L_j\varrho^{\kappa_i}, \quad i\in I(x).
	\end{equation} 
	Note that when $i \notin I(x)$,  we have already proved that $\lambda_i\leq \frac{M}{\delta'}$. As $\lambda\geq0$, the above inequality further indicates that 
	$\delta'\|\lambda\|_\infty\leq M+ 2L\varrho^\kappa+\sum_{j=1}^mL_j\varrho^{\kappa_i} \|\lambda\|_\infty$. Suppose the $\varrho$ in Assumption \ref{assum:upper} is taken small enough so that 
	\begin{equation}
		\label{eqn:gen-LICQ-varrho'}
		\delta'/2 \geq \sum_{j=1}^mL_j\varrho^{\kappa_i}.
	\end{equation}
	Then we obtain the following upper bound for the multipliers  
	$\|\lambda\|_\infty\leq 2\left(M+ 2L\varrho^\kappa\right)/\delta'.$
	That said, Assumption \ref{assum:upper} holds with any small enough $\varrho$ satisfying \eqref{eqn:gen-LICQ-varrho} and \eqref{eqn:gen-LICQ-varrho'}, and $B = 2\left(M+ 2L\varrho^\kappa\right)/\delta'.$
\end{proof}

\begin{theorem}
	Given any sufficiently small $\mu>0$ such that Proposition \ref{proposition:barrier} holds, suppose we run Algorithm CEB for $\big\lceil2\left(L+B\sum_{i=1}^m L_i\right)^2\frac{\Delta}{L\epsilon^2}\big\rceil$
	steps. Then there exists $1\leq k_* \leq K$ such that $\big\|\nabla 
	f(x^{k_*})+\sum_{i=1}^m \lambda_i \nabla c_i(x^{k_*})\big\|\leq \epsilon$, indicating an $O(\epsilon^{-2})$ iteration complexity for CEB.
\end{theorem}
\begin{proof}
	According to \eqref{eqn:dkmin0}, there exists an index $k_*$ s.t. 
	$\|x^{k_*}-x^{k_*-1}\|\leq \left(\frac{2\Delta}{LK}\right)^{1/2}\leq \varrho$ when CEB runs for $K\geq \frac{2\Delta}{L\varrho^2}$ steps. Then taking $k = k_*-1$ in \eqref{eqn0:grad-conti-kkt}  gives   
	\begin{eqnarray*}
		& & \Big\|\nabla 
		f(x^{k_*})+\sum_{i=1}^m \lambda_i \nabla c_i(x^{k_*})\Big\|\\
		& = &\Big\| \Big(L  + \sum_{i=1}^m \lambda_i L_i\Big) d^{k_*-1}+\Big(\nabla f(x^{k_*-1})-\nabla f(x^{k_*})\Big) 
		+ \sum_{i=1}^m \lambda_i\Big(\nabla c_i(x^{k_*})\Big)\Big\| \\
		&\leq & ( L+ \sum_{i=1}^m \lambda_i L_i) \big\|d^{k_*-1}\big\|\nonumber\\
		&\leq & \left(L+B\sum_{i=1}^m L_i\right)\left(\frac{2\Delta}{LK}\right)^{1/2},
	\end{eqnarray*}
	where we applied Proposition \ref{proposition:barrier} and \eqref{eqn:dkmin0}.
\end{proof}

\subsection{A Dikin Ellipsoid Approach}
In the case of Lipschitz gradient objective and constraints, the subproblems of GHMA are 
\begin{eqnarray}
	\label{prob:subproblem-Lip}
	& \min_{p\in\RR^n} & f(x_k) + \nabla f(x_k)^\top p + \frac{L}{2}\cdot\|p\|^{2} \\
	& \mathrm{s.t.} & c_i(x_k)+\nabla c_i(x_k)^\top p + \frac{L_i}{2}\cdot\|p\|^{2} \leq 0, \,\, \mbox{ for }\,\, i\in [m] . \nonumber
\end{eqnarray}
In case $m$ is large, the above subproblem may still be expensive to solve. Inspired by the Dikin ellipsoid approach in Algorithm FOSO, one may consider a Dikin ellipsoid as a surrogate for all the quadratic constraints, thus simplifying the subproblem considerably. 
Consider the self-concordant log-barrier function:
$$\tilde B_{k}(p):=-\sum_{i=1}^m\log\left(-c_i(x_k)-\nabla c_i(x_k)^\top p-\frac{L_i}{2}\|p\|^2\right).$$
Since $\tilde B_k$ is self concordant with self-concordance constant 1; see \cite{nesterov2018lectures}. In this case, the Dikin ellipsoid is defined by
\begin{equation}
	\label{defn:Dikin}
	\cE_k:=\left\{p\in\RR^n: p^\top H_k p\leq 1\right\}
\end{equation}
and 
\begin{equation}
	\label{defn:HessBarier}
	H_k:=\nabla^2 \tilde B_{k}(0) = \left(\sum_{i=1}^m\frac{L_i}{-c_i(x_k)}\right)\cdot I + \sum_{i=1}^m\frac{\nabla c_i(x_k)\nabla c_i(x_k)^\top}{c_i^2(x_k)}.
\end{equation} 
It is known that the Dikin ellipsoid $\cE_k$ is fully contained in the interior of the feasible region of \eqref{prob:subproblem-Lip}. Consequently, for  $\forall p\in\cE_k$, we have $c_i(x_k+p)<0$ for $i=1,...,m$. Therefore, we arrive at the following affine scaling alternative 
\begin{equation}
	\label{algo:Affine-Scaling}
	\begin{cases}
		p_k = \argmin_{p\in\cE_k} \,\, f(x_k) + \nabla f(x_k)^\top p + \frac{L}{2}\|p\|^2\\
		x_{k+1} = x_k + \alpha_kp_k,
	\end{cases}
\end{equation}
where $\alpha_k\in(0,1]$ is a diminishing stepsize that satisfies 
\begin{equation}
	\label{eqn:summable}
	\sum_{k=0}^{\infty}\alpha_k = +\infty \qquad\mbox{and}\qquad \sum_{k=0}^{\infty}\alpha_k^2 < +\infty.
\end{equation}
Although in practice taking $\alpha_k\equiv1$ often works very well, it is essential to take a diminishing stepsize. Otherwise, adversarial counterexamples exist where $x_k$ converges to non-stationary boundary points. We show how this is possible by the following example. \vspace{0.4 cm}

\shadowbox{\begin{minipage}{5.2in}
		{\bf Algorithm CEAS (Convex Enveloping with Affine Scaling)}
		
		\begin{description}
			
			\item[Step 0:] Choose $c_i(x^0) < 0$, $i\in [m]$. Let $k:=0$.
			
			\item[Step 1:] 
			Update by formula \eqref{algo:Affine-Scaling}. If $\|x^k-x^{k+1}\| > \epsilon$ then set $k:=k+1$, and return to {\bf Step 1}; else stop the algorithm. 
		\end{description}
\end{minipage}}  \vspace{0.4 cm}

Without diminishing stepsizes $\alpha_k$, Algorithm CEAS does not converge in general, as the example shown in Appendix \ref{cexample}.

\section{A Second-Order Extension} 
\label{Second_Order_Extension}

In this section, we introduce a two-tier  second-order approach to find approximate second-order solutions of the following constrained nonconvex optimization model
\[
\begin{array}{ll}
	\min & f(x)  \\
	\mbox{s.t.} & c_i(x) \le 0,\, i\in [m] \\
\end{array}
\]
where $f$ is gradient {\it and}\/ Hessian Lipschitz with constants $L$ and $L_{h_0}$ respectively, and $c_i$'s are gradient
Lipschitz with constants $L_i$, 
$i\in [m]$. This is a smooth special case of $(P)$ where $r(x)=0$ and $\cX = \mathbf{R}^n$. In this case, at iteration $k$, we apply a two-tier approach, namely we first apply Algorithm CMMA with 
\[
\tilde{f}(x\mid x^k ) = f(x^k) + \nabla f(x^k)^\top (x-x^k) + \frac{L\|x-x^k\|^2}{2} 
\]
and 
\[
\tilde{c}_i(x\mid x^k ) = c_i(x^k) + \nabla c_i(x^k)^\top (x-x^k) + \frac{L_i\|x-x^k\|^2}{2} ,\, i\in [m].
\]
We proceed by 
solving the first-order approximation subproblem:
\[
\begin{array}{lll}
	(FP_k) & \min & \nabla f(x^k)^\top (x-x^k) + \frac{L}{2} \| x - x^k\|^2 \\
	& \mbox{s.t.} & 
	c_i(x^k) + \nabla c_i(x^k)^\top (x-x^k) + \frac{L_i\|x-x^k\|^2}{2} \le 0,\, i\in [m] , \\
\end{array}
\]
and let its solution $x^{k+1}$ be the next iterate,   
until we find $x^{k}$ to be an approximate KKT point of $(FP_k)$, defined by the fact that the optimal value of the subproblem $(FP_k)$ is no less than $-\epsilon_1$, where $\epsilon_1>0$ is a prescribed precision. 
In this circumstance, we then resort to the following second-order surrogate subproblem with another prescribed precision $\epsilon_2>0$, where we aim to reduce the objective value by further exploiting the potential negative curvature directions of the Hessian matrix: 
\[
\begin{array}{lll}
	(SP_k) & \min & \frac{1}{2} (x-x^k)^\top
	\nabla^2 f(x^k) (x-x^k) + \frac{L_{h_0}}{6} \| x - x^k\|^3 \\
	& \mbox{s.t.} & 
	\nabla f(x^k)^\top (x-x^k) \le \epsilon_2 \\ 
	& & c_i(x^k) + \nabla c_i(x^k)^\top (x-x^k) + \frac{L_i\|x-x^k\|^2}{2} \le 0,\, i\in [m] . \\
\end{array}
\]
The above subproblem $(SP_k)$ has a nice convex constraint structure. 
Nonetheless, its objective is nonconvex, which makes it an NP-hard problem in general (see e.g.~\cite{NLR2018}). Observe, however, that the constraint set admits a self-concordant barrier function
\begin{eqnarray*} 
	B_k(x) &:=&  
	- \ln \left( \epsilon_2 - \nabla f(x^k)^\top (x-x^k) \right) \\
	& & 
	- \sum_{i=1}^m \ln \left( - c_i(x^k) - \nabla c_i(x^k)^\top (x-x^k) - \frac{L_i\|x-x^k\|^2}{2} \right) . 
\end{eqnarray*}
It is well-known (cf.~\cite{N2004}) that the so-called Dikin ellipsoid defined by the local norm induced by the Hessian of the self-concordant barrier $B_k(x)$ at $x^k$ is completely contained in the feasible region of $(SP_k)$. Therefore, we consider the following direction-finding subproblem
\[
\begin{array}{lll}
	(SP_k)' & \min & \frac{1}{2} d^\top \nabla^2 
	f(x^k) d + \frac{L_{h_0}}{6} \| d \|^3 \\
	& \mbox{s.t.} & d^\top \nabla^2 B_k(x^k) d \le \delta<1. 
\end{array}
\]
If $(SP_k)'$ is solvable with solution $d^k$, then we update the iterate as $x^{k+1}:=x^k+d^k$, 
and the iteration moves on with solving $(FP_{k+1})$ first before attempting $(SP_{k+1})'$. 
Denote the optimal value of $(SP_k)'$ to be $v((SP_k)')$. Then, 
\begin{eqnarray*}
	&&f(x^{k+1})\\
	&\le & \!
	f(x^k) \!+ \!\nabla f(x^k)^\top (x^{k+1}-x^k) \!+ \!\frac{1}{2} (x^{k+1}-x^k)^\top \nabla^2 f(x^k) (x^{k+1}-x^k) \!+ \!\frac{L_{h_0}}{6} \| x^{k+1}-x^k \|^3 \\
	&\le& f(x^k) + \epsilon_2 + v((SP_k)') . 
\end{eqnarray*}
If $v((SP_k)') < - 2 \epsilon_2$, then we have $f(x^{k+1}) < f(x^k) - \epsilon_2$, 
which assures a sufficient descent. Otherwise, it signifies the fact that $x^k$ is already an approximate stationary point in the second-order sense. 
Therefore, what remains is to study how $(SP_{k})'$ may be solved. To simplify the notation, consider a generic and purified form of $(SP_k)'$ as follows: 
\[
\begin{array}{lll}
	(SP)' & \min & d^\top Q d + \| d \|^3 \\
	& \mbox{s.t.} & d^\top P d \le 1, 
\end{array}
\]
where $P\succ 0$. Moving forward, we introduce a key ingredient known as Brickman's theorem (cf.~\cite{B1961,ML2005}), which states the following fact: 
\begin{proposition} (Brickman's Theorem) 
	Suppose $n\ge 3$. Let $Q_1$ and $Q_2$ be two real symmetric matrices. Then, 
	\[
	S(Q_1,Q_2):=
	\left. \left\{ \left( \begin{array}{c} x^\top Q_1 x \\ x^\top Q_2 x \end{array} \right) \, \right| \, x\in \RR^n \mbox{ with } \|x\|=1 \right\}
	\]
	is a convex set in $\RR^2$. 
\end{proposition}

Next, we shall give an LMI representation for the above convex set.
\begin{lemma}
	Suppose $n\ge 3$, and $Q_1,Q_2$ are two arbitrary real symmetric matrices. Then, 
	\[
	S(Q_1,Q_2)=
	\left. \left\{ \left( \begin{array}{c}  Q_1 \bullet X \\ Q_2 \bullet X \end{array} \right) \, \right| \, X\succeq 0 \mbox{ and } {\rm Tr}\, X = 1 \right\}.
	\]
\end{lemma}

\begin{proof}
	First of all, it is standard that 
	\[
	S(Q_1,Q_2)
	\subseteq 
	\left. \left\{ \left( \begin{array}{c}  Q_1 \bullet X \\ Q_2 \bullet X \end{array} \right) \, \right| \, X\succeq 0 \mbox{ and } {\rm Tr}\, X = 1 \right\}. 
	\]
	Now, take an arbitrary point 
	\[
	v \in 
	\left. \left\{ \left( \begin{array}{c}  Q_1 \bullet X \\ Q_2 \bullet X \end{array} \right) \, \right| \, X\succeq 0 \mbox{ and } {\rm Tr}\, X = 1 \right\}.
	\]
	That is, there exists $\hat X \succeq 0$ with $ \mbox{\rm Tr}\, \hat X =1$ and $v_i = Q_i \bullet \hat X$, $i=1,2$. By Proposition 3 in Sturm and Zhang~\cite{SZ2003}, one can compute a rank-one decomposition of $\hat X$ in such a way that
	\[
	\hat X = \sum_{i=1}^r \hat x_i \hat x_i^\top \mbox{ and }
	\|\hat x_i\|^2 = 1/r, \, i=1,2,...,r
	\]
	where $r=\mbox{\rm rank}\, \hat X$. 
	Therefore, 
	\[
	\left( \begin{array}{c} r \hat x_i^\top Q_1 \hat x_i \\ r \hat x_i^\top Q_2 \hat x_i \end{array} \right) \in S(Q_1,Q_2), \mbox{ with } i =1,2,...,r. 
	\]
	Since $S(Q_1,Q_2)$ is a convex set by Brickman's theorem, we have
	\[
	\frac{1}{r} \sum_{i=1}^r  \left( \begin{array}{c} r \hat x_i^\top Q_1 \hat x_i \\ r \hat x_i^\top Q_2 \hat x_i \end{array} \right) = 
	\left( \begin{array}{c}  Q_1 \bullet \hat X \\  Q_2 \bullet \hat X \end{array} \right) =  v \in S(Q_1,Q_2), 
	\]
	which proves the theorem. In other words, 
	$$S(Q_1,Q_2) = \left. \left\{ \left( \begin{array}{c}  Q_1 \bullet X \\ Q_2 \bullet X \end{array} \right) \, \right| \, X\succeq 0 \mbox{ and } {\rm Tr}\, X = 1 \right\}.$$
\end{proof} 
\noindent Denote 
\[
H(Q_1,Q_2) := \left. \left\{ \left( \begin{array}{c} x^\top x \\ x^\top Q_1 x \\ x^\top Q_2 x \end{array} \right) \, \right| \, x\in \RR^n \right\} \in \RR^3,
\]
which is a {\it homogenized}\/ conic version (aka perspective) of $S(Q_1,Q_2)$; see e.g.~\cite{BV2004}. Therefore, it is a convex cone if $n\ge 3$, thanks to Brickman's theorem ~\cite{BV2004}. Moreover, it has an SDP representation: 
\[
H(Q_1,Q_2) = \left. \left\{ \left( \begin{array}{c}  
	I \bullet X \\
	Q_1 \bullet X \\ 
	Q_2 \bullet X \end{array} \right) \, \right| \, X\succeq 0 \right\}. 
\]
Therefore, if $n\ge 3$, then solving $(SP)'$ is equivalent to the following convex SDP problem:
\begin{eqnarray}
	\label{subprob:SDP}
	& \min & y_3 + y_1^{3/2} \nonumber\\
	& \mbox{s.t.} & 
	\left( \begin{array}{c}  
		y_1 \\ y_2 \\ y_3  
	\end{array} \right) \in H(P,Q) \\
	& & y_2 \le 1, \nonumber
\end{eqnarray}

We call solution $x^k$ to be an $(\epsilon_1,\epsilon_2)$-approximate KKT solution ($\epsilon_1>0,\,\epsilon_2>0$) if the optimal value of $(FP_k)$ is no less than $-\epsilon_1$ {\it and}\/ the optimal value of $(SP_k)'$ is no less than $-2\epsilon_2$. That is, \ $ v((FP_k)) \ge -\epsilon_1$ and $v((SP_k)')\ge -\epsilon_2$. \vspace{0.4cm}

\shadowbox{\begin{minipage}{5.4in}
		{\bf Algorithm FOSO (First-Order and Second-Order)}

		\begin{description}
			
			\item[\textbf{Step 0}:] Choose $c_i(x^0) < 0$, $i\in [m]$. Let $k:=0$.
			
			\item[\textbf{Step 1}:] 
			Solve 
			\[
			\begin{array}{lll}
				(FP_k) & \min & \nabla f(x^k)^\top d + \frac{L}{2} \|d \|^2 \\
				& \mbox{s.t.} & 
				c_i(x^k) + \nabla c_i(x^k)^\top d + \frac{L_i\|d\|^2}{2} \le 0,\, 
				i \in [m]  \\
			\end{array}
			\]
			If $v((FP_k)) < - \epsilon_1$, then $x^{k+1}:=x^k+d^\star$, $k:=k+1$, and return to {\bf Step 1}. 
			
			\item[\textbf{Step 2}:] Else, solve 
			\[
			\begin{array}{lll}
				(SP_k)' & \min & \frac{1}{2} d^\top \nabla^2 f(x^k) d + \frac{L_{h_0}}{6} \| d \|^3 \\
				& \mbox{s.t.} & d^\top \nabla^2 B_k(x^k) d \le 1. 
			\end{array}
			\]
			If $v((SP_k)') < - 2 \epsilon_2$, then $x^{k+1}:=x^k+d^\star$, $k:=k+1$, and go to {\bf Step 1}. Else, stop.  
		\end{description}
\end{minipage}}
\vspace{0.3cm}

\noindent Based on the above derivation in this section, we obtain the following theorem. \vspace{-0.3cm}
\begin{theorem}
	The above described two-tier approach Algorithm FOSO produces a sequence $\{x^k :\: k=1,2,...\}$ that is monotonically decreasing in the objective value, and the process stops in no more than $O(1/\epsilon_1)$ steps of solving first-order subroutines $(FP_k)$ and no more than $O(1/\epsilon_2)$ steps of solving second-order subroutines $(SP_k)'$ before reaching an $(\epsilon_1,\epsilon_2)$-approximate KKT solution. 
\end{theorem}

As far as we know, the above is a first result along the lines of reaching a second-order stationary solution for nonconvex optimization with nonconvex constraints. For nonconvex optimization with convex constraints, there are a few recent papers aiming to reach a second-order stationary solution; see~\cite{HL2023,DS2024,WWGLZ2024}. Moreover, in these approaches, the Hessian Lipschitz property with respect to the local norm defined by the self-concordant barrier function is assumed. That assumption, though plausible is hard to verify in practice. A nice feature of the above proposed approach is that we no longer need that assumption, in addition to the feature that the constraint set is no longer assumed to be convex.

\section{Numerical Experiments} \label{numerical} 
In this section, we test the numerical performance of GHMA, CEB, and CEAS with barrier algorithms on several smooth and nonsmooth constrained nonconvex optimization problems. In the later discussion, we will omit the ``with barrier'' terms and directly use CEAS for simplicity. For CEB and CEAS methods, their subproblems are solved via first order methods, and they will be skipped in nonsmooth test problems because these two algorithms are specifically designed for smooth problems. For all three methods, their algorithmic parameters are selected exactly according to the theory, with the involved Lipschitz/H\"older constants obtained by simple numerical estimation. We take the fmincon solver in Matlab as the standard benchmark to compare. To maintain the feasibility of the iterations, we let fmincon execute the interior point algorithm.

\subsection{Fair machine learning}
First, to test the general nonsmooth setting, we consider the following fair machine learning problem and solve it with GHMA algorithm.   We consider a sparse logistic regression setting with a dataset of $n$ observations $\{(a_i, b_i)\}_{i=1}^n$, where $a_i \in \mathbb{R}^d$ represents the $d$-dimensional feature vector for the $i$-th observation, and $b_i \in \{0, 1\}$ represents the binary responses. Then the log-likelihood function will be 
\[
\ell(\beta) = \frac{1}{n}\sum_{i=1}^{n} \left[ b_i \log\left(\frac{1}{1 + \exp(-a_i^\top \beta)}\right) + (1 - b_i) \log\left(1 - \frac{1}{1 + \exp(-a_i^\top \beta)}\right) \right],
\]
where $\beta \in \mathbb{R}^d$ are the parameters to be estimated. Suppose the $n$ data points actually consist of $K$ groups $C_k$, $k\in[K]$ with significantly different sizes $|C_k|$ and moderately different ground truth weight  parameters $\beta_k$, that is, each $b_i\sim\mathrm{Bernoulli}(\sigma(a_i^\top\beta_k))$ if data $i$ belongs to the group $C_k$, where $\sigma(\cdot)$ is the sigmoid function.

Because the sizes $|C_k|$ are very unbalanced, on the one hand, directly maximizing $\ell(\beta)$ will inarguably bias towards the majority group, while the under-represented groups with relatively small sample sizes will suffer. The best solution would be training the regression model separately for each group $C_k$ and keep the fitted weight $\hat{\beta}_k$. Then if a data point comes from group $C_k$, we use the corresponding weight $\hat{\beta}_k$ to predict the response.  Yet, on the other hand, the privacy concerns may usually prevent many actual users to reveal their true group identities, and therefore, a unified model that performs equally well for all groups would still be desirable. Therefore, we may also compute the log-likelihood function of each class
\[
\ell_j(\beta) =  \frac{1}{|C_j|}\sum_{i\in C_j} \left[ b_i \log\left(\frac{1}{1 + \exp(-a_i^\top \beta)}\right) + (1 - b_i) \log\left(1 - \frac{1}{1 + \exp(-a_i^\top \beta)}\right) \right], j\in[K]
\]
and formulate the following fairness constrained optimization problem  
\begin{eqnarray}
	\label{prob:fair}
	&\min_{x} &-\ell(\beta) + \mu\|\beta\|_1\nonumber\\
	&\textrm{s.t.} & -\ell_i(\beta)\leq -(1+\eta)\ell_j(\beta), \quad 1\leq i\neq j \leq K.
\end{eqnarray}
We should note that, even though $\ell$ is convex, the constraints are nonconvex. 

We set $\eta=0.05$ such that the performance measure, in terms of log-likelihood function, differs for at most 5\% between different groups. 
And we set $K=2$ and $n=4000$ in the experiments. For the 2 groups, the minority group $C_1$ accounts for 20\% training data, and the majority group $C_2$ accounts for 80\% training data. For feature dimension, we test the cases with dimension $d\in\{200,1000\}$. For sparsity level, we test the cases with $\mu\in\{2^{-7}, 2^{-8}, 2^{-9}, 2^{-10}\}$. Due to the nonsmoothness, CEB, CEAS, and fmincon do not work for this instance, we only test GHMA and plot the function value gap and KKT-violation in Figure \ref{fair_fig}. Also because the curves are for different problem sizes and parameter settings under the same algorithm GHMA, plotting running time will not bring much insight. Therefore, we will use \#iterations as the x-axis of the figures.   
\begin{figure}[h] 
	\centering
	\begin{subfigure}{.255\linewidth}
		\includegraphics[width=\linewidth]{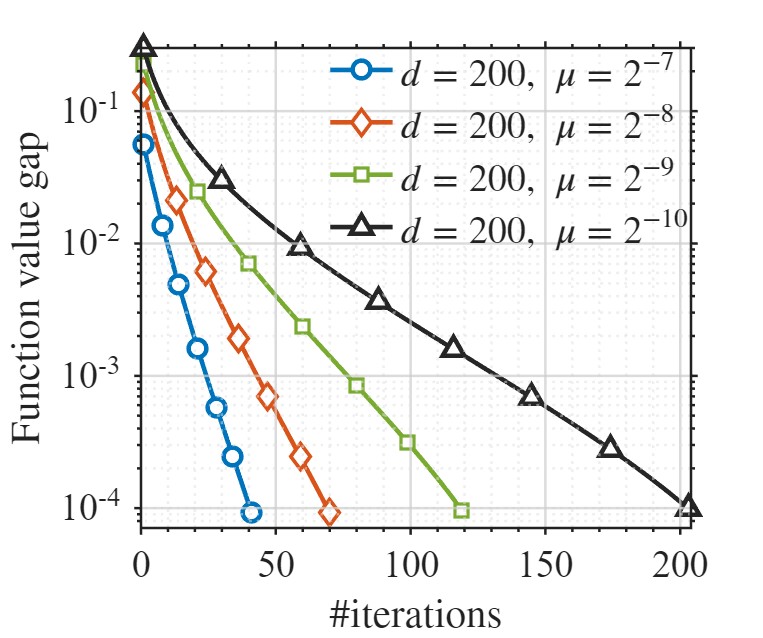}  
	\end{subfigure}\hspace{-0.25cm} 
	\begin{subfigure}{.255\linewidth}
		\includegraphics[width=\linewidth]{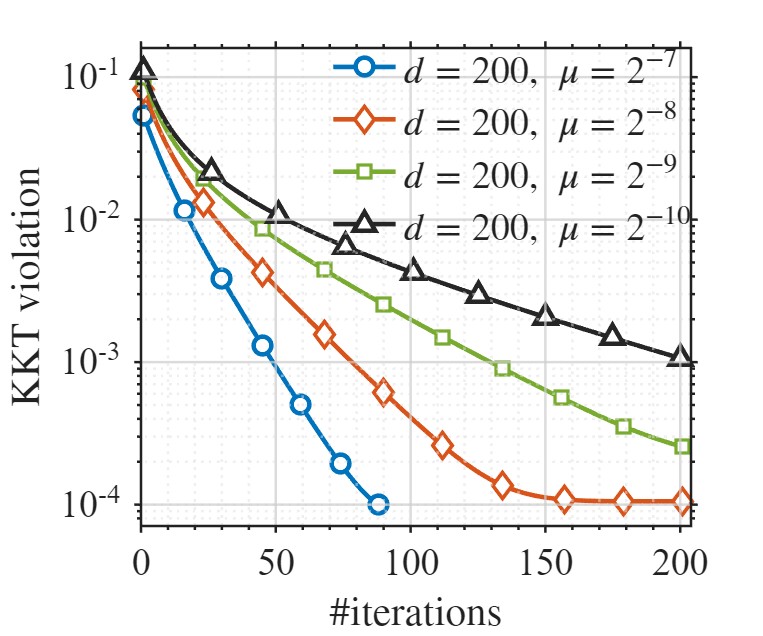} 
	\end{subfigure}\hspace{-0.25cm}
	\begin{subfigure}{.255\linewidth}
		\includegraphics[width=\linewidth]{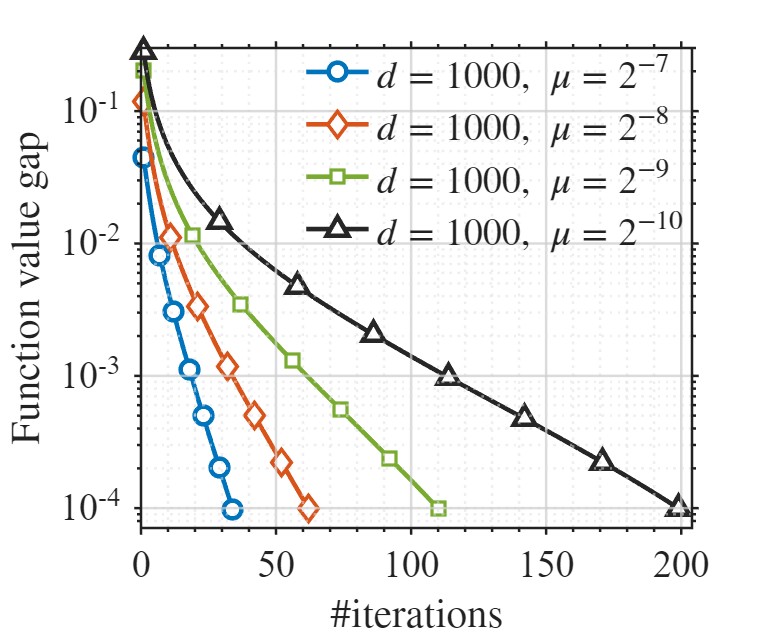}  
	\end{subfigure}\hspace{-0.25cm}
	\begin{subfigure}{.255\linewidth}
		\includegraphics[width=\linewidth]{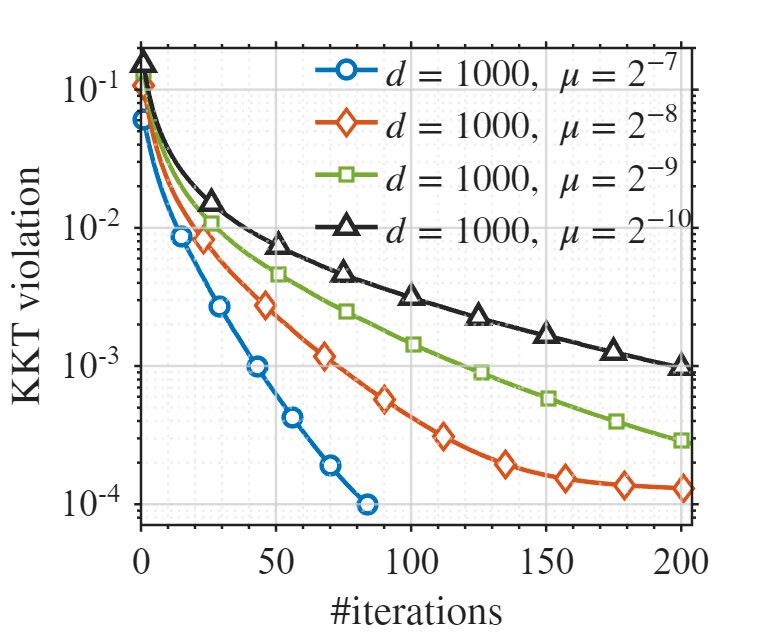} 
	\end{subfigure}
	
	\caption{The performance of GHMA under different sizes and sparsity parameters.}
	\label{fair_fig}
\end{figure}
Because the problem is nonconvex, the globally optimal solution is not known. Therefore, we compute the function value gap w.r.t. the solution output by GHMA for sufficiently long time with several different initializations. 

Finally, to illustrate the effect of the fairness constraints in problem \eqref{prob:fair}, we present the accuracy rates of the output $\hat{\beta}$ among $C_1$ and $C_2$, and compare it to the output of standard logistic regression without fairness constraints.  As the behavior of the instances $d=200$ and $d=1000$ are similar, we omit the $d=200$ instance. 

\begin{table}[h]
	\centering
	\caption{The effect of fairness constraints, under instance $d=1000$.  }
	\renewcommand{\arraystretch}{1.18}
	\setlength{\tabcolsep}{15pt}
	\begin{tabular}{c cc cc}
		\toprule
		& \multicolumn{2}{c}{With fairness constraints} 
		& \multicolumn{2}{c}{Without fairness constraints} \\
		\cmidrule(lr){2-3} \cmidrule(lr){4-5}
		$\mu$ & $C_1$ acc. & $C_2$ acc. & $C_1$ acc. & $C_2$ acc. \\
		\midrule
		$2^{-7}$  & 83.50\% & 86.31\% & 78.12\% & 87.31\% \\
		$2^{-8}$  & 91.50\% & 93.69\% & 87.12\% & 94.78\% \\
		$2^{-9}$  & 96.63\% & 96.72\% & 92.63\% & 97.97\% \\
		$2^{-10}$ & 98.12\% & 98.22\% & 95.15\% & 99.19\% \\
		\bottomrule
	\end{tabular} 
	\label{tab:fairness_acc}
\end{table}

\subsection{Copositive programming} 
To incorporate the other algorithms and the fmincon benchmark, we choose (the nonconvex and low-rank approximation of) the copositive programming problem that we mention in the introduction as our second test example:
\[\begin{aligned}
	\min_{X,Y} \quad & \langle C_0,X\rangle\\
	\textrm{s.t.} \quad & \langle C_i,X\rangle \le b_i, i=1,2...,m\\
	&X=YY^\top,Y\ge 0,
\end{aligned}
\]
where $X\in \mathbb{R}^{n\times n}$, $Y\in \mathbb{R}^{n\times r}$, $C_i\in \mathbb{R}^{n\times n}$ for $i=0,1,2,...,m$, and $C_0\succeq 0$. In the algorithmic implementation, we will automatically eliminate the $X$ variable by substituting $X = YY^\top$ to the objective function and constraints. We set $n=100$, $r=60$ and $m=200$, which gives problems with 6000 variables and 200 constraints. And we generate three random instances for numerical tests. Figure \ref{compare_all} provides the performance of GHMA, CEB, and CEAS, with the standard Matlab nonlinear programming solver fmincon being the benchmark.  Similar to the fair machine learning instance, we compute the function value gap w.r.t. the solution output by the best performing algorithm (GHMA) for sufficiently long time. For the KKT violation plot, we will omit the curve of fmincon as this solver does not have the option to record this quantity along the iterations. Because the subproblems of different algorithms are significantly different, plotting figures w.r.t. iterations will not be informative. Therefore, we will plot the curves w.r.t. running time in this experiment. The results are provided in Figure~\ref{compare_all}.

\begin{figure}[h]
	\centering
	% first row: objective gap
	\begin{subfigure}{.3\linewidth}
		\includegraphics[width=\linewidth]{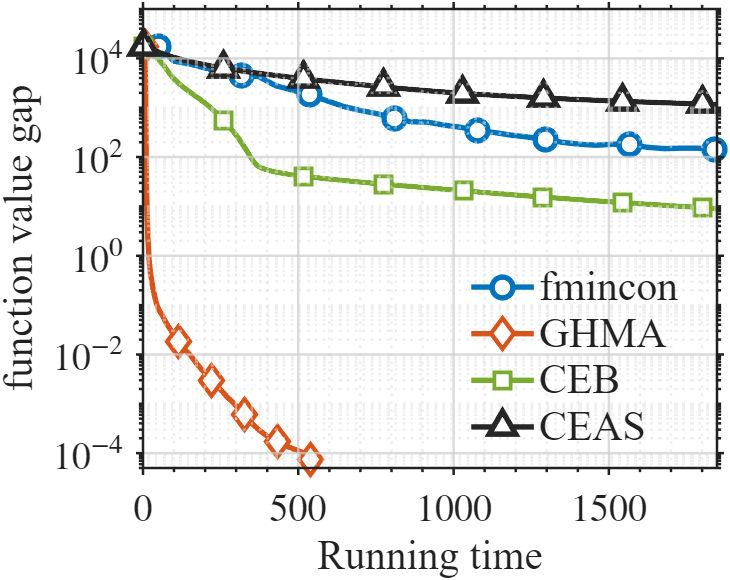}
	\end{subfigure}\hspace{0.35cm}
	\begin{subfigure}{.3\linewidth}
		\includegraphics[width=\linewidth]{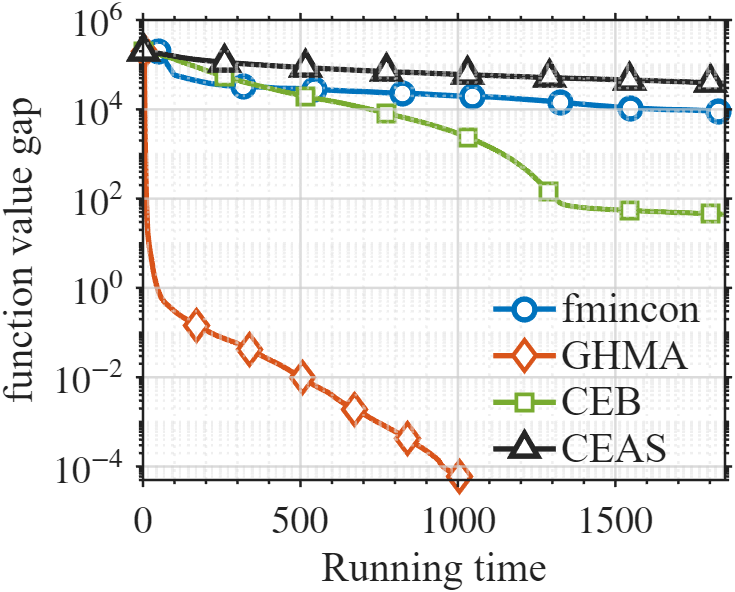}
	\end{subfigure}\hspace{0.35cm}
	\begin{subfigure}{.3\linewidth}
		\includegraphics[width=\linewidth]{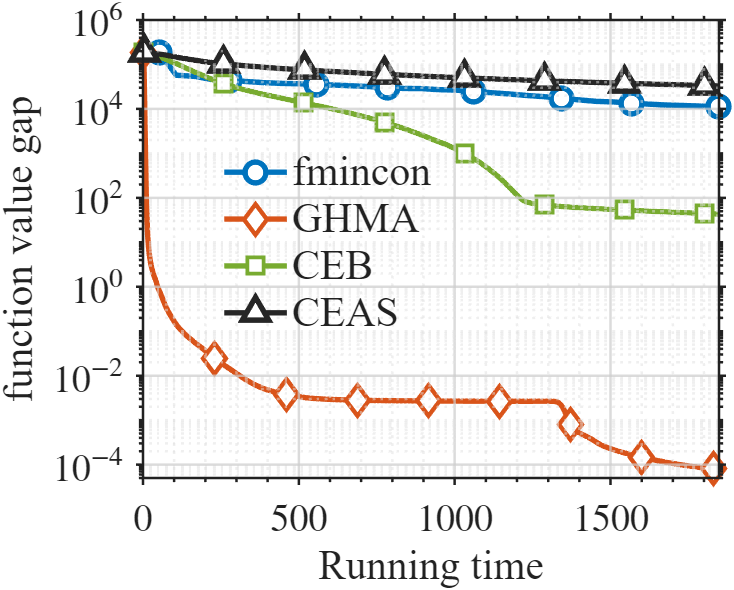}
	\end{subfigure}
	
	\vspace{0.25cm}
	
	% second row: KKT violation
	\begin{subfigure}{.3\linewidth}
		\includegraphics[width=\linewidth]{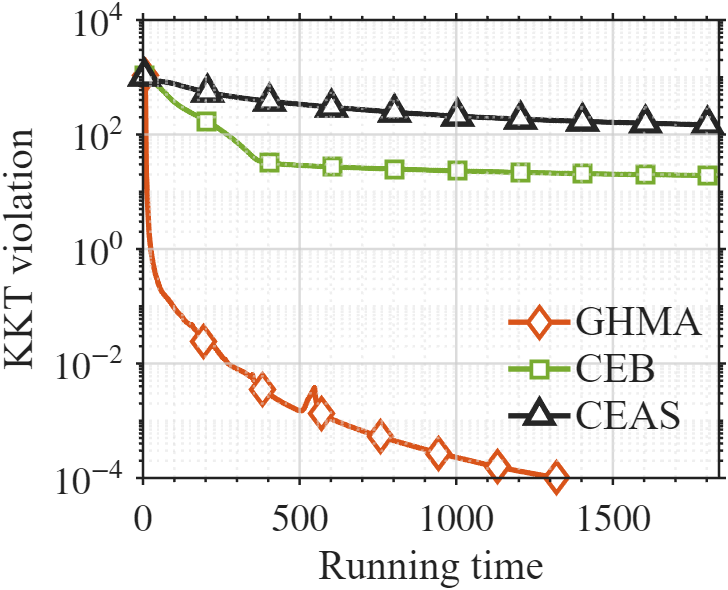}
	\end{subfigure}\hspace{0.35cm}
	\begin{subfigure}{.3\linewidth}
		\includegraphics[width=\linewidth]{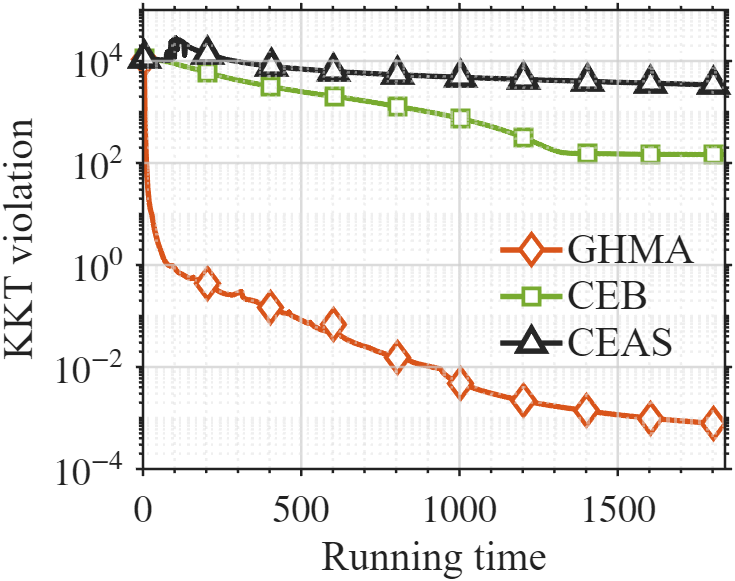}
	\end{subfigure}\hspace{0.35cm}
	\begin{subfigure}{.3\linewidth}
		\includegraphics[width=\linewidth]{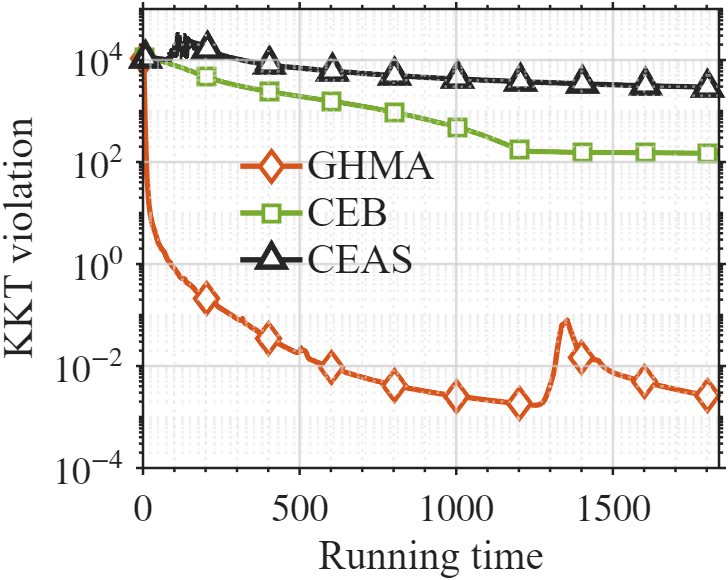}
	\end{subfigure}
	
	\caption{The function value gap and KKT violation curves of the synthesized instances.  }
	\label{compare_all}
\end{figure}

\subsection{Copositive relaxation of stable set problem} 
Finally, we test (the nonconvex and low-rank approximation of) the copositive relaxation of the stable set problem, which is a (non-synthesized) special case of the general copositive programming problem tested in the previous subsection. As discussed in the introduction, several NP-hard optimization problems can
be written as linear programs over the convex cone of copositive matrices; see e.g.~\cite{burer2009copositive}. Define the set of completely positive matrix as 
\[
\mathcal{C}_n^*=\bigg\{X\in \mathbb{R}^{n\times n}: X=YY^\top, \mbox{ for } Y\in\mathbb{R}^{n\times r}, Y\geq0, r=n(n+1)/2\bigg\}, 
\]
where $r=n(n+1)/2$ is suggested by the Carath\'eodory theorem \cite{rockafellar1997convex}. Let $\mathcal{G}=(\mathcal{V},\mathcal{E})$ be a graph with vertex set $\mathcal{V}$ and edge set $\mathcal{E}$. For an $n$-node graph with $\mathcal{V}=[n]$, it is known (see~\cite{de2002approximation}) that the maximum stable set number $\alpha^*$ for $\mathcal{G}$ satisfies 
\[\alpha^*=\max \left\{\langle e e^\top,X\rangle: X \in \mathcal{C}_n^*, \mathrm{Tr}(X)=1, X_{i j}=0, \forall(i, j) \in \mathcal{E}\right\} ,\]
where $e$ stands for the vector of all ones. 
Then it is equivalent to setting $\delta=0$ for 
\[\alpha^*(\delta)=\max \left\{\langle e e^\top,X\rangle: X \in \mathcal{C}_n^*, \mathrm{Tr}(X)\leq 1, X_{i j}\leq \delta, \forall(i, j) \in \mathcal{E}\right\} .\]
We test the algorithms on cyclic graph $\mathcal{G}_n$ with $n$ vertices and relax $\delta=10^{-4}$. In addition, we adopt the nonconvex rank-2 approximation of the problem to obtain the formulation:
\[\begin{aligned}
	\max_{X,Y} \quad & \langle ee^\top, X\rangle\\
	\textrm{s.t.} \quad & X=YY^\top,\,\,\,\, Y\in\mathbb{R}^{n\times 2} \\
	&Y\ge 0, \,\,\,\, \mathrm{Tr}(X)\leq 1\\
	& X_{ij}\leq \delta, \,\,\,\, \forall (i,j)\in\mathcal{E}
\end{aligned}
\] 
Similar to the general copositive program experiments, in the algorithmic implementation, we eliminate the $X$ variable. Because the running time for different algorithms are significantly different in this set of experiments, we choose to report the results by table. For the cyclic graph $\mathcal{G}_n$ that we test, its maximum stable set number admits a closed form solution. Therefore, to test the efficiency of each algorithm, we run each algorithm for a maximum of 30 seconds, and report the time among the three trials at which the algorithm reaches a sufficiently small objective function value gap $|\langle e e^\top,X\rangle-\alpha^*|<0.1$ in Table \ref{tab:StableNumber}. Because $\alpha^*$ is known to be an integer, we only need to select a moderate tolerance of $0.1$. If the algorithm does not reach the tolerance in 30 seconds, we report ``---'' in the table.

\begin{table}[t]
	\centering
	\caption{Running time of tested algorithms, under instances $\mathcal{G}_{10}$--$\mathcal{G}_{40}$.}
	\label{tab:StableNumber}
	\renewcommand{\arraystretch}{1.15}
	\setlength{\tabcolsep}{17pt}
	\begin{tabular}{c cccc}
		\toprule
		& \multicolumn{4}{c}{The tested algorithms} \\
		\cmidrule(lr){2-5}
		Graph & \texttt{fmincon} & GHMA & CEB & CEAS \\
		\midrule
		$\mathcal{G}_{10}$ & ${0.0364}\,\mathrm{s}$ & ${0.0041}\,\mathrm{s}$ & $23.4347\,\mathrm{s}$ & $ 3.7439\,\mathrm{s}$ \\
		$\mathcal{G}_{20}$ & $0.0740\,\mathrm{s}$ & ${0.0635}\,\mathrm{s}$ & $ 13.1789\,\mathrm{s}$ & $0.9943\,\mathrm{s}$ \\
		$\mathcal{G}_{30}$ & $0.2491 \,\mathrm{s}$ & ${0.0450}\,\mathrm{s}$ & \textemdash & \textemdash \\
		$\mathcal{G}_{40}$ & $0.5253\,\mathrm{s}$ & ${0.1068}\,\mathrm{s}$ & \textemdash & $ 5.7729\,\mathrm{s}$ \\
		\bottomrule
	\end{tabular} 
\end{table}

\section{Conclusions}
\label{conclusion}

In this paper, we proposed and studied the iteration complexity properties of a convex majorization scheme for nonconvex constrained optimization models. Specific implementations of the convex majorization scheme have been considered under suitable constraint qualifications, pending the gradient H\"{o}lderian continuous properties of the objective in question, or the availability of the higher-order information thereof. Our numerical experiments suggest practical efficacy of the proposed scheme and its various implementations, alongside the guaranteed theoretical iteration complexity bounds. 

%\backmatter

\iffalse
\section*{Declarations}

\noindent{\bf Funding acknowledgment. } 
Junyu Zhang is supported by the 
MOE AcRF grant A-0009530-05-00; Shuzhong Zhang is supported by the NSF grant CMMI-2228034. \vspace{0.3cm}

\noindent{\bf Competing interests.} This work does not have any competing interests. 
\vspace{0.3cm}

\noindent {\bf  Data availability.}
No external datasets were used or analyzed in this study. All problem instances reported in the numerical experiments were generated by the code developed for this work.\vspace{0.3cm}\\
\noindent{\bf Code availability.}
The code supporting the findings of this study is available from the first author upon reasonable request.
\fi

\bigskip

\bibliographystyle{plain}
\bibliography{sample}% common bib file

\begin{thebibliography}{10}

\bibitem{auslender2010moving}
Alfred Auslender, Ron Shefi, and Marc Teboulle.
\newblock A moving balls approximation method for a class of smooth constrained
  minimization problems.
\newblock {\em SIAM Journal on Optimization}, 20(6):3232--3259, 2010.

\bibitem{bai2022achieving}
Qinbo Bai, Amrit~Singh Bedi, Mridul Agarwal, Alec Koppel, and Vaneet Aggarwal.
\newblock Achieving zero constraint violation for constrained reinforcement
  learning via primal-dual approach.
\newblock In {\em Proceedings of the AAAI Conference on Artificial
  Intelligence}, volume~36, pages 3682--3689, 2022.

\bibitem{BGLSS21}
Antoine Bernigaud, Serge Gratton, Flavia Lenti, Ehouarn Simon, and Oumaima
  Sohab.
\newblock $l_p$-norm regularization approaches in variational data
  assimilation.
\newblock {\em Quarterly Journal of the Royal Meteorological Society},
  147:2067--2081, 2021.

\bibitem{bertsekas2014constrained}
Dimitri Bertsekas.
\newblock Constrained optimization and {L}agrange multiplier methods.
\newblock Academic press, (2014).

\bibitem{bertsekas1982projected}
Dimitri Bertsekas.
\newblock Projected {N}ewton methods for optimization problems with simple
  constraints.
\newblock {\em SIAM Journal on Control and Optimization}, 20(2):221--246, 1982.

\bibitem{bertsekas1997nonlinear}
Dimitri Bertsekas.
\newblock Nonlinear programming.
\newblock {\em Journal of the Operational Research Society}, 48(3):334--334,
  1997.

\bibitem{bolte2017multiproximal}
J{\'e}r{\^o}me Bolte, Zheng Chen, and Edouard Pauwels.
\newblock The multiproximal linearization method for convex composite problems.
\newblock {\em arXiv preprint arXiv:1712.02623}, 2017.

\bibitem{bolte2016majorization}
J{\'e}r{\^o}me Bolte and Edouard Pauwels.
\newblock Majorization-minimization procedures and convergence of sqp methods
  for semi-algebraic and tame programs.
\newblock {\em Mathematics of Operations Research}, 41(2):442--465, 2016.

\bibitem{boob2023stochastic}
Digvijay Boob, Qi~Deng, and Guanghui Lan.
\newblock Stochastic first-order methods for convex and nonconvex functional
  constrained optimization.
\newblock {\em Mathematical Programming}, 197(1):215--279, 2023.

\bibitem{BV2004}
Stephen Boyd and Lieven Vandenberghe.
\newblock Convex optimization.
\newblock Cambridge University Press, (2004).

\bibitem{B1961}
Louis Brickman.
\newblock On the field of values of a matrix.
\newblock {\em Proceedings of the American Mathematical Society}, 12:61--66,
  1961.

\bibitem{burer2009copositive}
Samuel Burer.
\newblock On the copositive representation of binary and continuous nonconvex
  quadratic programs.
\newblock {\em Mathematical Programming}, 120(2):479--495, 2009.

\bibitem{cartis2011evaluation}
Coralia Cartis, Nicholas~IM Gould, and Philippe~L Toint.
\newblock On the evaluation complexity of composite function minimization with
  applications to nonconvex nonlinear programming.
\newblock {\em SIAM Journal on Optimization}, 21(4):1721--1739, 2011.

\bibitem{cartis2014complexity}
Coralia Cartis, Nicholas~IM Gould, and Philippe~L Toint.
\newblock On the complexity of finding first-order critical points in
  constrained nonlinear optimization.
\newblock {\em Mathematical Programming}, 144(1):93--106, 2014.

\bibitem{chen2022near}
Fan Chen, Junyu Zhang, and Zaiwen Wen.
\newblock A near-optimal primal-dual method for off-policy learning in {CMDP}.
\newblock {\em Advances in Neural Information Processing Systems},
  35:10521--10532, 2022.

\bibitem{chow2019lyapunov}
Yinlam Chow, Ofir Nachum, Aleksandra Faust, Edgar Duenez-Guzman, and Mohammad
  Ghavamzadeh.
\newblock Lyapunov-based safe policy optimization for continuous control.
\newblock {\em arXiv preprint arXiv:1901.10031}, 2019.

\bibitem{CP2022}
Ying Cui and Jong-Shi Pang.
\newblock Modern nonconvex nondifferentiable optimization.
\newblock MOS-SIAM Series on Optimization, (2022).

\bibitem{de2002approximation}
Etienne De~Klerk and Dmitrii~V Pasechnik.
\newblock Approximation of the stability number of a graph via copositive
  programming.
\newblock {\em SIAM Journal on Optimization}, 12(4):875--892, 2002.

\bibitem{de2023constrained}
Alberto De~Marchi, Xiaoxi Jia, Christian Kanzow, and Patrick Mehlitz.
\newblock Constrained composite optimization and augmented {L}agrangian
  methods.
\newblock {\em Mathematical Programming}, 201(1):863--896, 2023.

\bibitem{deng2025augmented}
Kangkang Deng, Rui Wang, Zhenyuan Zhu, Junyu Zhang, and Zaiwen Wen.
\newblock The augmented lagrangian methods: Overview and recent advances.
\newblock {\em arXiv preprint arXiv:2510.16827}, 2025.

\bibitem{D2010}
Mirjam Dur.
\newblock Copositive programming – a survey.
\newblock In M.~Diehl, F.~Glineur, E.~Jarlebring, and W.~Michiels, editors,
  {\em Recent Advances in Optimization and its Applications in Engineering},
  pages 293--298. Springer, Berlin, 2010.

\bibitem{DS2024}
Pavel Dvurechensky and Mathias Staudigl.
\newblock Barrier {A}lgorithms for {C}onstrained {N}on-{C}onvex {O}ptimization.
\newblock {\em arXiv preprint arXiv:2404.18724}, 2024.

\bibitem{fiedler2024safety}
Christian Fiedler, Johanna Menn, Lukas Kreisk{\"o}ther, and Sebastian Trimpe.
\newblock On safety in safe bayesian optimization.
\newblock {\em arXiv preprint arXiv:2403.12948}, 2024.

\bibitem{grapiglia2021complexity}
Geovani~Nunes Grapiglia and Ya-xiang Yuan.
\newblock On the complexity of an augmented {L}agrangian method for nonconvex
  optimization.
\newblock {\em IMA Journal of Numerical Analysis}, 41(2):1546--1568, 2021.

\bibitem{guo2023safenon}
Baiwei Guo, Yuning Jiang, Giancarlo Ferrari-Trecate, and Maryam Kamgarpour.
\newblock Safe zeroth-order optimization using quadratic local approximations.
\newblock {\em arXiv preprint arXiv:2303.16659}, 2023.

\bibitem{guo2023safe}
Baiwei Guo, Yuning Jiang, Maryam Kamgarpour, and Giancarlo Ferrari-Trecate.
\newblock Safe zeroth-order convex optimization using quadratic local
  approximations.
\newblock In {\em 2023 European Control Conference (ECC)}, pages 1--8. IEEE,
  2023.

\bibitem{guo2023safeLP}
Baiwei Guo, Yang Wang, Yuning Jiang, Maryam Kamgarpour, and Giancarlo
  Ferrari-Trecate.
\newblock Safe zeroth-order optimization using linear programs.
\newblock {\em arXiv preprint arXiv:2304.01797}, 2023.

\bibitem{HL2023}
Chuan He and Zhaosong Lu.
\newblock {A} {N}ewton-{CG} based barrier method for finding a second-order
  stationary point of nonconvex conic optimization with complexity guarantees.
\newblock {\em SIAM Journal on Optimization}, 33(2):1191--1222, 2023.

\bibitem{li2021rate}
Zichong Li, Pin-Yu Chen, Sijia Liu, Songtao Lu, and Yangyang Xu.
\newblock Rate-improved inexact augmented {L}agrangian method for constrained
  nonconvex optimization.
\newblock In {\em International Conference on Artificial Intelligence and
  Statistics}, pages 2170--2178. PMLR, 2021.

\bibitem{lin2022complexity}
Qihang Lin, Runchao Ma, and Yangyang Xu.
\newblock Complexity of an inexact proximal-point penalty method for
  constrained smooth non-convex optimization.
\newblock {\em Computational Optimization and Applications}, 82(1):175--224,
  2022.

\bibitem{ML2005}
Juan~Enrique Martinez-Legaz.
\newblock On {B}rickman's theorem.
\newblock {\em Journal of Convex Analysis}, 12:139--143, 2005.

\bibitem{NN2024}
Yassine Nabou and Ion Necoara.
\newblock Moving higher-order {T}aylor approximations method for smooth
  constrained minimization problems.
\newblock {\em SIAM Journal on Optimization}, 36, 2024.

\bibitem{N2004}
Yurii Nesterov.
\newblock Introductory lectures on convex optimization: A basic course.
\newblock Kluwer Academic Publishers, (2004).

\bibitem{nesterov2018lectures}
Yurii Nesterov.
\newblock Lectures on {C}onvex {O}ptimization.
\newblock Springer, (2018).

\bibitem{nocedal1999numerical}
Jorge Nocedal and Stephen~J Wright.
\newblock Numerical optimization.
\newblock Springer, (1999).

\bibitem{NLR2018}
Maher Nouiehed, Jason Lee, and Meisam Razaviyayn.
\newblock Convergence to second-order stationarity for constrained non-convex
  optimization.
\newblock {\em arXiv preprint arXiv:1810.02024}, 2018.

\bibitem{nutalapati2019constrained}
Mohan~Krishna Nutalapati, Muppavaram~Sai Krishna, Atanu Samanta, and Ketan
  Rajawat.
\newblock Constrained non-convex optimization via stochastic variance reduced
  approximations.
\newblock In {\em 2019 Sixth Indian Control Conference (ICC)}, pages 293--298.
  IEEE, 2019.

\bibitem{rockafellar1997convex}
R~Tyrrell Rockafellar.
\newblock {\em Convex analysis}, volume~28.
\newblock Princeton university press, Princeton, NJ, 1997.

\bibitem{sahin2019inexact}
Mehmet~Fatih Sahin, Ahmet Alacaoglu, Fabian Latorre, Volkan Cevher, et~al.
\newblock An inexact augmented lagrangian framework for nonconvex optimization
  with nonlinear constraints.
\newblock {\em Advances in Neural Information Processing Systems}, 32, 2019.

\bibitem{SZ2003}
Jos Sturm and Shuzhong Zhang.
\newblock On cones of nonnegative quadratic functions.
\newblock {\em Mathematics of Operations Research}, 28:246--267, 2003.

\bibitem{wei2025sqp}
Pin-Zheng Wei and Wei-Hong Yang.
\newblock An sqp-type proximal gradient method for constrained composite
  optimization.
\newblock {\em Journal of the Operations Research Society of China}, pages
  1--25, 2025.

\bibitem{WWGLZ2024}
Chenyu Wu, Nuozhou Wang, Casey Garner, Kevin Leder, and Shuzhong Zhang.
\newblock Novel optimization techniques for parameter estimation.
\newblock {\em arXiv preprint arXiv:2407.04235}, 2024.

\bibitem{zhao2023state}
Weiye Zhao, Tairan He, Rui Chen, Tianhao Wei, and Changliu Liu.
\newblock State-wise safe reinforcement learning: A survey.
\newblock {\em arXiv preprint arXiv:2302.03122}, 2023.

\bibitem{ZZZ24}
Daoli Zhu, Lei Zhao, and Shuzhong Zhang.
\newblock A first-order primal-dual method for nonconvex constrained
  optimization based on the augmented {L}agrangian.
\newblock {\em Mathematics of Operations Research}, 49(1):1--651, 2024.

\end{thebibliography}

\appendix

	\section{CEAS Counterexample}
	\label{cexample}
	
	\begin{example}
		\label{example:diminishing}
		Consider the objective function $f(z) :=  \arctan\left(x/y\right)$ defined on the set $\mathcal{X} =\left\{z=(x,y): x>0, \|z\|\geq1/2\right\}$\footnote{As the constructed example will never exit $\mathcal{X}$, one can arbitrarily and smoothly extend the value of $f$ beyond $\mathcal{X}$.}, and $c(z):=\|z\|^2-1$.  If we run 
		Algorithm CEAS 
		to solve
		$$\min_z f(z)\,\,\,\,\mathrm{s.t.}\,\,\,\,c(z)\leq 0$$ with $\alpha_k\equiv1$ and initial point $z_0 = (1-\epsilon_0)(\sin\theta_0,-\cos\theta_0)$. Then $\{z_k\}$ converges geometrically to a non-stationary boundary point for any $\theta_0<\pi$, as long as $\epsilon_0\leq\min\{1/32,(\pi-\theta_0)^2/100\}$. 
	\end{example}
	\begin{proof}
		First, the function is essentially the negative radian of the angle between the vector $z=(x,y)$ and $(0,-1)$; see Figure~\ref{fig:example:diminishing}. By direct computation, we have 
		$\nabla f(z) = \frac{(y,-x)^\top}{\|z\|^2}$ and $\nabla c(z) = 2z$. In addition, we also have $\|\nabla^2 f(z)\|\leq \frac{\sqrt{2}}{\|z\|^2}\leq 4\sqrt{2}$ for all $z\in\mathcal{X}$. Then it is sufficient to set $L = 4\sqrt{2}$ and $L_1 = 2$ when running algorithm \eqref{algo:Affine-Scaling}. 
		Also note that $\nabla f(z)$ is always perpendicular to $\nabla c(z)$. Consequently, the two axes of the Dikin ellipsoid 
		$$\cE_k:=\left\{p\in\RR^2:p^\top\left(\frac{2}{-c(z_k)}\cdot I + \frac{4}{c^2(z_k)}\cdot z_kz_k^\top\right)p\leq 1\right\}$$
		will align with $\nabla f(z)$ and $\nabla c(z)$, respectively.  
		
		Now consider the iteration $z_k = (1-\epsilon_k)\cdot(\sin\theta_k,-\cos\theta_k)$. Then $\frac{\|\nabla f(z_k)\|}{L} = \frac{1}{4\sqrt{2}(1-\epsilon_k)}\geq \sqrt{\frac{-c(z_k)}{2}}$ as long as  $\epsilon_k\leq 1/32$, where $-c(z_k) = 1-(1-\epsilon_k)^2 = 2\epsilon_k-\epsilon_k^2$. That is, $-\nabla f(z_k)/L$ is longer than the axis aligning with it, and is always out of $\cE_k$, see the second subfigure of Figure \ref{fig:example:diminishing}. Consequently, $z_{k+1}$ will be the end point of the long axis. Hence
		$\|z_{k+1}-z_k\|^2 = \frac{-c(z_k)}{2} = \epsilon_k-\frac{\epsilon_k^2}{2}$ and   
		$$\epsilon_{k+1} = 1-\sqrt{\|z_k\|^2 + \|z_k-z_{k+1}\|^2} = 1-\sqrt{1-\epsilon_k+\epsilon_k^2/2}\leq \frac{16}{25}\epsilon_k$$
		as long as $\epsilon_k\leq 1/32$. Meanwhile, in terms of the radian increment, we have
		$$\Delta\theta_k \leq \tan\Delta\theta_k = \frac{\|z_k-z_{k+1}\|}{\|z_k\|} = \frac{\sqrt{\epsilon_k-\frac{\epsilon_k^2}{2}}}{1-\epsilon_k} < 2\sqrt{\epsilon}_k$$
		as long as $\epsilon_k\leq 1/32$. Combined with $\sqrt{\epsilon_{k+1}}\leq \frac{4}{5}\sqrt{\epsilon_k}$, we have 
		$$\sum_{k=0}^{\infty}\Delta\theta_k < 2\sum_{k=0}(4/5)^k\sqrt{\epsilon_0}\leq 10\sqrt{\epsilon_0}.$$
		Therefore, for any $\theta_0<\pi$, as long as we choose $\epsilon_0\leq\min\{1/32,(\pi-\theta_0)^2/100\}$, the sequence $z_k = (1-\epsilon_k)\cdot(\sin\theta_k,-\cos\theta_k)$ converges   geometrically to some $\bar{z} = (\sin\bar{\theta},-\cos\bar{\theta})$, with $\bar{\theta}<\pi,$ which is not a stationary point of the problem. A numerical illustration of this phenomena can be found in Figure \ref{fig:example:diminishing}, which can be resolved by adopting a diminishing stepsize. 
		\begin{figure}[htbp]
			\centering
			\begin{subfigure}{0.224\textwidth}
				\centering \includegraphics[width=\linewidth]{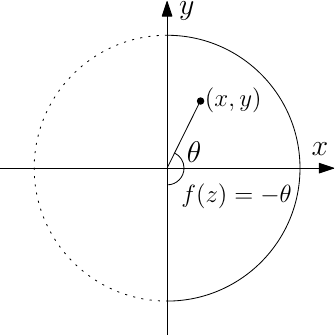}
				\caption{Objective function}
			\end{subfigure}
			\hfill
			\begin{subfigure}{0.24\textwidth}
				\centering \includegraphics[width=\linewidth]{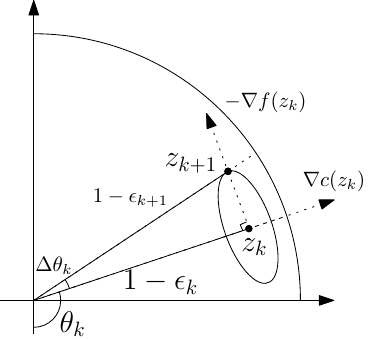}
				\caption{One iteration}
			\end{subfigure}
			\hfill
			\begin{subfigure}{0.24\textwidth}
				\centering \includegraphics[width=\linewidth]{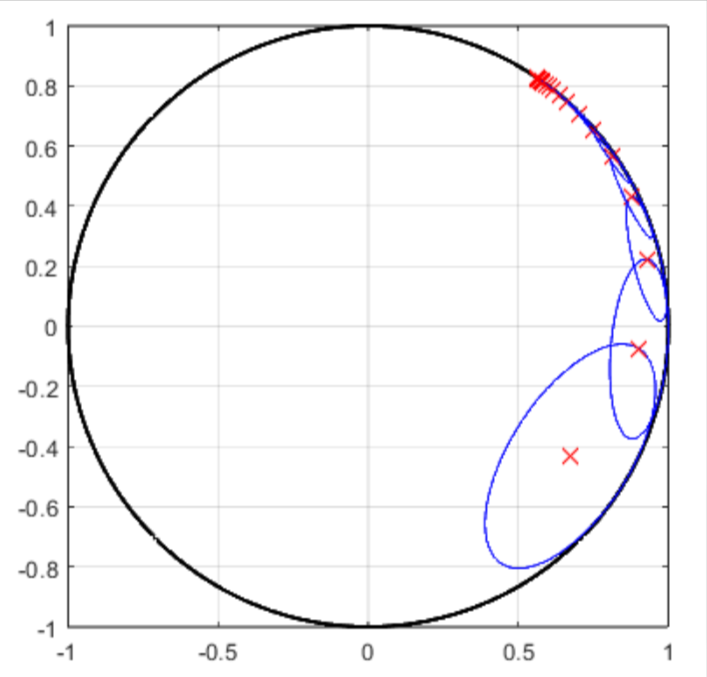}
				\caption{Constant $\alpha_k\equiv1$}
			\end{subfigure}
			\hfill
			\begin{subfigure}{0.24\textwidth}
				\centering \includegraphics[width=\linewidth]{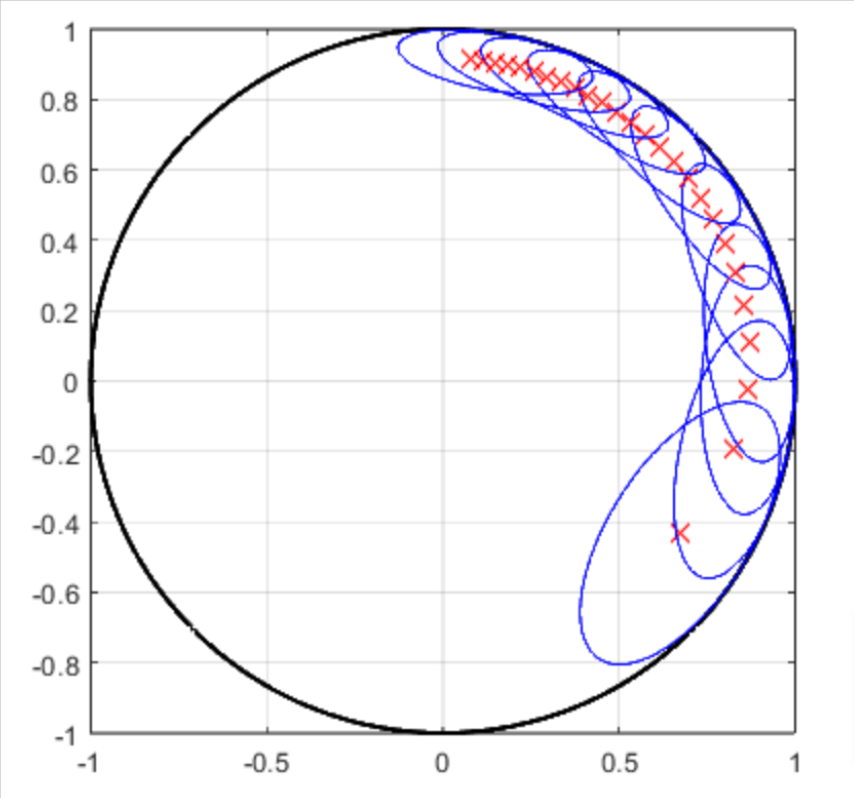}
				\caption{Diminishing $\alpha_k = \frac{1}{\sqrt{k}}$}
			\end{subfigure}
			\caption{Graphical illustration of Example \ref{example:diminishing}}
			\label{fig:example:diminishing}
		\end{figure} 
	\end{proof}
	
	\begin{theorem}
		\label{theorem:affine-scaling}
		Let $\{x_k\}_{k\geq0}$ be generated by 
		Algorithm CEAS. 
		If there exists $x^*$ s.t. $x_k \to x^*$ as $k\to\infty$, then $x^*$ is feasible and there exists $\lambda^*$ s.t. $\nabla f(x^*) + \sum_{i=1}^m\lambda^*_i\nabla c_i(x^*) = 0$, and we have $\lambda_i^*c_i(x^*)=0$ for all $1\le i \le m$. 
	\end{theorem}

	We provide the following roadmap for the proof, for ease of referencing.

	\begin{figure}[h]
		\centering
		\includegraphics[width=1\linewidth]{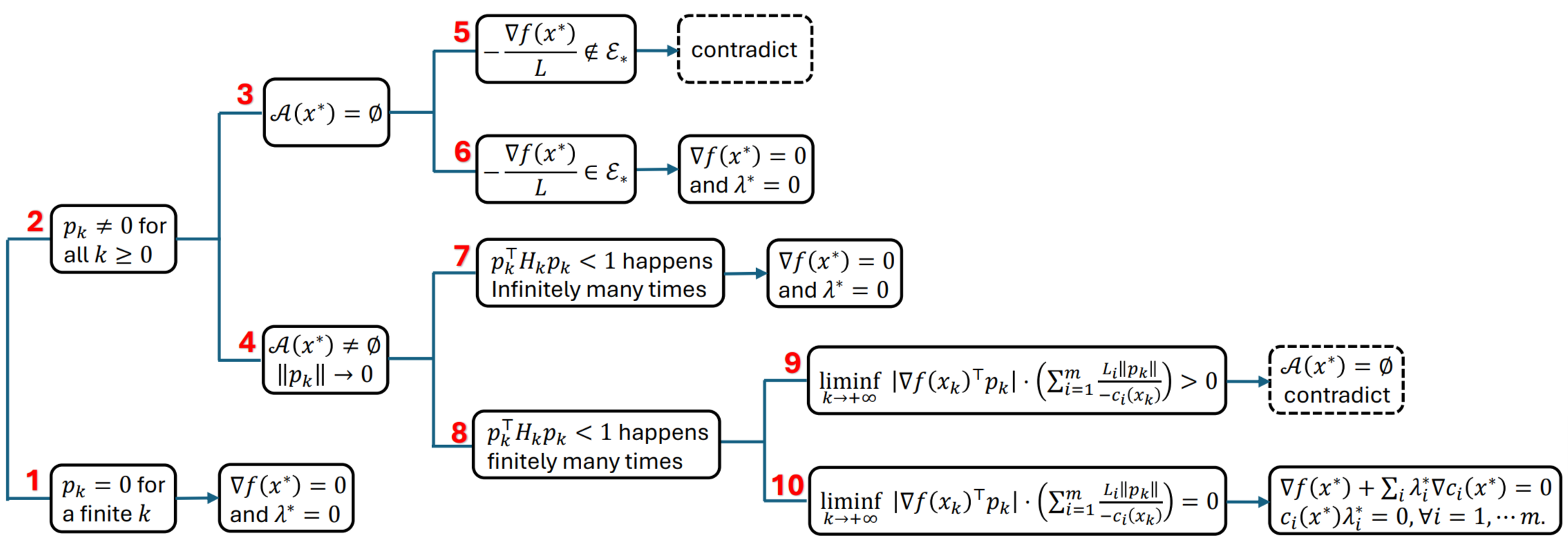}
		\caption{Proof sketch for Theorem \ref{theorem:affine-scaling}} 
	\end{figure}

	\begin{proof}
		First of all, the optimality condition of \eqref{algo:Affine-Scaling} gives  
		\begin{equation}
			\label{thm:AFS-0}
			\begin{cases}
				\nabla f(x_k) + Lp_k + \xi_kH_kp_k = 0,\\
				p_k^\top H_kp_k \leq 1,\qquad\qquad\,\, \xi_k\geq0,\\
				\xi_k(1-p_k^\top H_kp_k) = 0.
			\end{cases}
		\end{equation}
		To continue the discussion, let us clear up a few corner cases. 
		
		\textbf{Branch 1.} First, if $p_k=0$ for some finite $k_0\geq0$, then \eqref{thm:AFS-0} immediately indicates that $\nabla f(x_{k_0})=0$ then Theorem \ref{theorem:affine-scaling} holds trivially with $x^*=x_{k_0}$ and $\lambda^*=0$. Therefore, without loss of generality, we can assume $\nabla f(x^*)\neq0$ and $p_k\neq0$ for all $k\geq 0$ in the following discussion.
		
		\textbf{Branch 2.} As $p_k\neq0$ and $H_k\succ0$, we have $p_k^\top H_kp_k>0$. Then left multiplying $p_k^\top$ to the first equation of \eqref{thm:AFS-0} and rearranging the terms gives 
		\begin{equation}
			\label{thm:AFS-1}
			\xi_k = -\frac{\nabla f(x_k)^\top p_k+L\|p_k\|^2}{p_k^\top H_kp_k} \geq 0,
		\end{equation}
		and hence $\nabla f(x_k)^\top p_k \leq -L\|p_k\|^2$. Meanwhile, by the descent lemma, we have 
		\begin{eqnarray*}
			f(x_{k+1}) & \leq & f(x_k) + \alpha_k\nabla f(x_k)^\top p_k + \frac{L\alpha_k^2}{2}\|p_k\|^2\\
			& = & f(x_k) + \frac{\alpha_k}{2}\nabla f(x_k)^\top p_k + \frac{\alpha_k}{2}\nabla f(x_k)^\top p_k + \frac{L\alpha_k^2}{2}\|p_k\|^2\\
			& \overset{(i)}{\leq} & f(x_k) + \frac{\alpha_k}{2}\nabla f(x_k)^\top p_k - \frac{L\alpha_k(1-\alpha_k)}{2}\|p_k\|^2\\
			& \leq & f(x_k) + \frac{\alpha_k}{2}\nabla f(x_k)^\top p_k\\
			& \overset{(ii)}{\leq} & f(x_k) - \frac{L\alpha_k}{2}\|p_k\|^2\,.
		\end{eqnarray*}
		Both (i) and (ii) are done by applying $\nabla f(x_k)^\top p_k \leq -L\|p_k\|^2$ to the term $\frac{\alpha_k}{2}\nabla f(x_k)^\top p_k$. Telescoping the last two inequalities gives the following two bounds
		\begin{equation}
			\label{thm:AFS-2}
			\sum_{k=0}^{\infty} -\alpha_k\nabla f(x_k)^\top p_k \leq \frac{2\Delta_f}{L}\qquad\mbox{and}\qquad\sum_{k=0}^{\infty} \alpha_k\|p_k\|^2 \leq \frac{2\Delta_f}{L} 
		\end{equation}
		where $\Delta_f:=f(x_0)-f^*$ and $f^*$ is the optimal value. Both bounds will be crucial in the later analysis. In particular, as $\sum_{i=0}^{\infty}\alpha_k=+\infty$, the second bound indicates that $\liminf_{k\to\infty}\|p_k\|=0$. 
		
		Our next step is to prove that $\lim_{k\to\infty}\|p_k\|=0$. And this proof will consist of two cases, based on whether the active set  $\cA(x^*) \neq \emptyset$ or $\cA(x^*) = \emptyset$.

		\textbf{Branch 3.} Suppose $\cA(x^*) = \emptyset$. In this case, $c_i(x^*)<0$ for all $i=1,2,\cdots,m$. Note that the subproblem is in fact a projection problem
		$$p^*(x):=\argmin_p \,\,\Big\|p+\frac{1}{L}\nabla f(x)\Big\|^2\qquad\mbox{s.t.}\qquad p^\top\nabla^2h_{x}(0)p\leq 1$$
		parameterized by the iterate $x$. With a slight abuse of notation, we simply denote $p_k=p^*(x_k)$ and $p_*=p(x^*)$.  Note that objective function and constraint set are continuous in terms of the parameter $x$. Therefore, Berge's maximum theorem holds and the optimal value is a continuous function of $x$, consequently the assumption $x_k\to x^*$ gives 
		\begin{equation}
			\label{thm:AFS-3}
			\lim_{k\to\infty} \Big\|p_k+\frac{1}{L}\nabla f(x_k)\Big\|^2 = \Big\|p_{*}+\frac{1}{L}\nabla f(x^*)\Big\|^2.
		\end{equation}
		\textbf{Branch 5.} If $\big\|p_{*}+\frac{1}{L}\nabla f(x^*)\big\|^2>0$, then there exist $k_0\geq0$ such that $\big\|p_{k}+\frac{1}{L}\nabla f(x_k)\big\|^2>0$ for all $k\geq0$. In projection problems, this means that $-\frac{1}{L}\nabla f(x_k)\notin\cE_k$ and the constraint $p_k\in\mathcal{E}_k$ must be active for all $k\geq k_1$. Also note that 
		$$\lim_{k\to\infty}H_k  = \left(\sum_{i=1}^m\frac{L_i}{-c_i(x_k)}\right)\cdot I + \sum_{i=1}^m\frac{\nabla c_i(x_k)\nabla c_i(x_k)^\top}{c_i^2(x_k)}\prec+\infty$$ is a finite matrix when $\cA(x^*)=\emptyset$, there exists some constant $C_1$ s.t. $\|H_k\|\leq C_1$ for all $k\geq k_1$. Therefore, we must have 
		$1 = p_k^\top H_k p_k \leq C_1\|p_k\|^2$ for all $k\geq k_1$. That is $\|p_k\|^2 \geq C_1^{-1}$ for all $k\geq k_1$. By \eqref{eqn:summable}, we have 
		\begin{equation}
			\label{thm:AFS-4}
			\sum_{k=0}^{\infty}\alpha_k\|p_k\|^2 \geq  \sum_{k=k_1}^{\infty}\alpha_k\|p_k\|^2 \geq \sum_{k=k_1}^{\infty}\alpha_k/C_1 = +\infty,
		\end{equation}
		contradicting the summability result \eqref{thm:AFS-1}. 
		
		\textbf{Branch 6.} Due to the above discussion, we must have $\big\|p_{k}+\frac{1}{L}\nabla f(x^k)\big\|^2\to\big\|p_{*}+\frac{1}{L}\nabla f(x^*)\big\|^2=0$. 
		Consequently,
		$$\lim_{k\to\infty} \big\|p_{k}+\frac{1}{L}\nabla f(x^*)\big\| \leq \lim_{k\to\infty}\Big(\|p_{k}+\frac{1}{L}\nabla f(x^k)\| + \|\nabla f(x_k) - \nabla f(x^*)\|\Big)=0,$$ 
		indicating that
		$\lim_{k\to\infty}  p_{k} = -\frac{1}{L}\nabla f(x^*).$
		If $\nabla f(x^*)\neq0$, there exists some $k_2\geq0$ and $C_2>0$ such that $\|p_k\|\geq C_2$ for all $k\geq k_2$. Similar to the argument of \eqref{thm:AFS-4}, this again yields a contradiction with the summability result \eqref{thm:AFS-2}. That is $\|p_k\|\to0$ as $k\to\infty$, and $\nabla f(x^*)=0$. Then this theorem holds trivially by setting $\lambda^*=0$.

		\textbf{Branch 4.} Suppose $\cA(x^*) \neq \emptyset$.  By \eqref{defn:HessBarier}, $c_i(x_k)\to 0$ for $i\in\cA(x^*)$ indicates that 
		$$H_k\succ\left(\sum_{i=1}^m\frac{L_i}{-c_i(x_k)}\right)\cdot I \to \infty.$$
		Hence the constraint $p_k^\top H_k p_k\leq 1$ indicates $\|p_k\|^2 \leq {1}/{\sum_{i=1}^m\frac{L_i}{-c_i(x_k)}}\to0$. 
		
		\textbf{Branch 7.} Now we can continue the discussion conditioning on $\cA(x^*)\neq\emptyset$ and $\lim_{k\to\infty}\|p_k\|=0$. Next, we would like to discuss based on whether there is an infinite subsequence of subproblems where the constraints $p\in\cE_k$ are inactive.  
		
		\textbf{Branch 8.} Suppose there is an infinite subsequence $\KK_1\subseteq\mathbb{Z}_+$ such that $p\in\cE_k$ is inactive. That is, $p_k^\top H_k p_k <1$ for all $k\in\KK_1$. Combined with the complementary slackness condition in \eqref{thm:AFS-0} indicates that $\xi_k = 0$ for $k\in\KK_1$. Consequently, we have 
		$\nabla f(x_k) + Lp_k = 0$ for $\forall k\in\KK_1$. As $x_k\to x^*$ and $p_k\to0$ as $k\to\infty$, we have 
		$$\nabla f(x^*) = \lim_{k\to\infty}\nabla f(x_k) = \lim_{k\in\KK_1}\nabla f(x_k) = \lim_{k\in\KK_1} -Lp_k = 0.$$
		We again have $\nabla f(x^*)=0$ in this case and the theorem holds by setting $\lambda^*=0$. 
		
		Therefore, without loss of generality, we can further assume that there exists some $k_3\geq0$ such that the subproblem constraints $p\in\cE_k$ are active for all $k\geq k_3$. Namely, we have 
		\begin{equation}
			\label{thm:AFS-5}
			p_k^\top H_k p_k = 1 \qquad\mbox{for}\qquad \forall k\geq k_3.
		\end{equation}
		Substituting \eqref{thm:AFS-5} to \eqref{thm:AFS-1} gives
		\begin{equation}
			\label{thm:AFS-6}
			\xi_k = - \big(\nabla f(x_k) + L p_k \big)^\top p_k \qquad\mbox{for}\qquad \forall k\geq k_3.
		\end{equation}
		Denote the unit update direction vector in the iteration $k$ as 
		$$\bar{p}_k:=p_k/\|p_k\|.$$
		Then substituting \eqref{thm:AFS-6} and \eqref{defn:HessBarier} to \eqref{thm:AFS-0} and rearranging the terms gives
		\begin{equation}
			\label{thm:AFS-7}
			\nabla f(x_k)  + \sum_{i=1}^m \lambda_{k,i}\cdot\nabla c_i(x_k) + R_k\cdot\bar{p}_k + Lp_k = 0
		\end{equation}
		where the approximate multiplier $\lambda_{k,i}$ and the error term $R_k$ are given by 
		\begin{eqnarray}
			\label{thm:AFS-8}
			\lambda_{k,i} &:= & - \big(\nabla f(x_k) + L p_k \big)^\top \bar{p}_k\cdot \frac{\|p_k\|^2}{c_i^2(x_k)}\cdot\bar{p}_k^\top\nabla c_i(x_k), \qquad i = 1,2,\cdots,m,\\
			R_k & := &  - \big(\nabla f(x_k) + L p_k \big)^\top \bar{p}_k\cdot \left(\sum_{i=1}^m\frac{L_i\|p_k\|^2}{-c_i(x_k)}\right). 
		\end{eqnarray}
		Clearly, \eqref{thm:AFS-1} indicates that $R_k\geq0$. 
		As $\|p_k\|\to0$, and $x_k\to x^*$, there exists constant $C_3>0$ such that $\big\|\nabla f(x_k) + L p_k\big\|\leq C_3$ and $\|\nabla c_i(x_k)\|\leq C_3$ for all $k\geq0$. Therefore, for all $i\notin\cA(x^*)$ s.t. $c_i(x^*)\neq0$, we have 
		\begin{equation}
			\label{thm:AFS-9}
			\lim_{k\to\infty}|\lambda_{k,i}| \leq C_3^2\cdot\lim_{k\to\infty} \frac{\|p_k\|^2}{c_i^2(x_k)} = \frac{C_3^2}{c_i^2(x^*)}\cdot\lim_{k\to\infty}\|p_k\|^2 =0\quad\mbox{for}\quad i\notin\cA(x^*).
		\end{equation}
		Intuitively, as the term $Lp_k\to0$, if we  further have $R_k\to0$, and $\lambda_{k,i}\to\lambda^*_i$ for $i\in\cA(x^*)$, then the theorem is proved. However, this proof is very tricky. To establish the theorem, we consider the following two cases. 
		
		\textbf{Branch 9.} First, let us consider the following situation where
		\begin{equation}
			\label{thm:AFS-fake}
			\liminf_{k\to\infty} R_k>0.
		\end{equation} 
		We aim to show by contradiction that this case cannot happen. In this case, there exist a constant $C_4>0$ and some $k_4\geq k_3$ (see \eqref{thm:AFS-5} for definition of $k_3$) such that $R_k\geq C_4$ for all $k\geq k_4$. Note that the feasibility $p_k\in\cE_k$ indicates that
		$$\sum_{i=1}^m\frac{L_i\|p_k\|^2}{-c_i(x_k)}\leq p_k^\top H_kp_k\leq 1\quad\mbox{for}\quad\forall k\geq0,$$
		this immediately indicates 
		$R_k\leq C_3$ for $k\geq k_4$. In addition, we further have for all $k\geq k_4$ that 
		\begin{equation}
			\label{thm:AFS-10}
			\begin{cases}
				-(\nabla f(x_k) + L p_k)^\top \bar{p}_k = \frac{R_k}{\sum_{i=1}^m\frac{L_i\|p_k\|^2}{-c_i(x_k)}} \geq C_4,\vspace{0.1cm}\\
				\sum_{i=1}^m\frac{L_i\|p_k\|^2}{-c_i(x_k)} = \frac{R_k} {-(\nabla f(x_k) + L p_k )^\top \bar{p}_k}  \geq {C_4}/{C_3}.
			\end{cases}
		\end{equation} 
		Substituting the first inequality of \eqref{thm:AFS-10} to the summability bound \eqref{thm:AFS-2} gives 
		$$\frac{2\Delta_f}{L} \geq \sum_{k=0}^{\infty} -\alpha_k\nabla f(x_k)^\top p_k \geq  \sum_{k=k_4}^{\infty} -\nabla f(x_k)^\top \bar{p}_k\cdot \alpha_k\|p_k\|\geq C_4\sum_{k=k_4}^{\infty}\alpha_k\|p_k\|.$$
		This argument provides us a much stronger summability bound:
		\begin{equation}
			\label{thm:AFS-11}
			\sum_{k=k_4}^{\infty}\alpha_k\|p_k\| \leq \frac{2\Delta_f}{LC_4}.
		\end{equation}
		On the other hand, by LICQ, $\{\nabla c_i(x^*)\}_{i\in\cA(x^*)}$ are linearly independent. As $x_k\to x^*$, we know there exists $k_5\geq k_4$ such that $\{\nabla c_i(x_k)\}_{i\in\cA(x^*)}$ are linearly independent for all $k\geq k_5$. Define 
		$$\begin{cases}
			J_{k,A}:=\begin{bmatrix}
				\nabla c_i(x_k)^\top
			\end{bmatrix}_{i\in\cA(x^*)}\\
			J_{k,N}:=\begin{bmatrix}
				\nabla c_i(x_k)^\top
			\end{bmatrix}_{i\notin\cA(x^*)}
		\end{cases}\qquad\mbox{and}\qquad\quad
		\begin{cases}
			\lambda_{k,A}:=\begin{bmatrix}
				\lambda_{k,i}
			\end{bmatrix}_{i\in\cA(x^*)}\\
			\lambda_{k,N}:=\begin{bmatrix}
				\lambda_{k,i}
			\end{bmatrix}_{i\notin\cA(x^*)}
		\end{cases}$$
		Similarly, we define the notations $J_{*,A}$, $J_{*,N}$ by replacing $x_k$ with $x^*$ in  the above definition.   Then \eqref{thm:AFS-7} indicates for $k\geq k_5$ that   
		\begin{equation}
			\label{thm:AFS-lamA}
			\lambda_{k,A} = -\left[J_{k,A}J_{k,A}^\top\right]^{-1}\!\!J_{k,A}\left(\nabla f(x_k) +    J_{k,N}^\top\lambda_{k,N} + R_k\cdot\bar{p}_k + Lp_k\right).
		\end{equation}
		Note that in the above equation, all the term on the right hand side are bounded for all $k\geq k_5$:
		$$J_{k,N}^\top\lambda_{k,N}\to0,\qquad\,\|\nabla f(x_k) + Lp_k\|\leq C_3,\qquad \qquad \qquad \!$$ 
		$$\|R_k\cdot\bar{p}_k\| \leq C_3, \qquad \left[J_{k,A}J_{k,A}^\top\right]^{-1}\!\!J_{k,A}\to \left[J_{*,A}J_{*,A}^\top\right]^{-1}\!\!J_{*,A}.$$
		Therefore, $\lambda_{k,A}$ must be bounded for all $k$. Therefore, we can suppose there exists $\Lambda >0$ such that 
		\begin{equation}
			\label{thm:AFS-12}
			|\lambda_{k,i}|\leq \Lambda\quad\mbox{for}\quad\forall k\geq0, \forall i\in\cA(x^*).
		\end{equation}
		Now we finish the preparation and can start proving the contradiction by showing that $\cA(x^*)=\emptyset$. By the $L_i$-smoothness of $c_i(\cdot)$, we have the following bounds for the margin of constraints:
		\begin{equation}
			\label{thm:AFS-13}
			-c_i(x_{k+1})\geq -c_i(x_k) - \alpha_k\nabla c_i(x_k)^\top p_k - \frac{L_i\alpha^2_k}{2}\|p_k\|^2.
		\end{equation}
		Based on this, we derive two-types of descent for $-c_i(\cdot)$. First, the feasibility of $p_k$ gives
		$$\sum_{i=1}^m\frac{L_i\|p_k\|^2}{-c_i(x_k)} + \sum_{i=1}^m\frac{(\nabla p_k^\top c_i(x_k))^2}{c_i^2(x_k)}\leq 1.$$
		That is, $\frac{L_i}{2}\|p_k\|^2\leq -c_i(x_k)$ and $|\nabla c_i(x_k)^\top p_k|\leq -c_i(x_k)$ for all $k\geq 0$ and all $i=1,2,\cdots,m.$ Substituting them to \eqref{thm:AFS-13} gives the type-1 descent inequality:
		\begin{equation}
			\label{thm:AFS-14}
			-c_i(x_{k+1})\geq -c_i(x_k)\left(1-\alpha_k-\frac{\alpha_k^2}{2}\right).
		\end{equation}
		On the other hand, the upper bound \eqref{thm:AFS-12} and the definition \eqref{thm:AFS-8} of $\lambda_{k,i}$ indicates that 
		\begin{eqnarray*}
			\frac{\big|\bar{p}_k^\top\nabla c_i(x_k)\big|}{-c_i(x_k)} & = & \frac{|\lambda_{k,i}|}{\big|\big(\nabla f(x_k) + L p_k \big)^\top \bar{p}_k\big| \cdot \frac{\|p_k\|^2}{-c_i(x_k)}}\\
			& \leq & \frac{\Lambda}{C_4 \cdot \frac{\|p_k\|^2}{-c_i(x_k)}}\\
			& = & \frac{\Lambda}{C_4 \cdot \frac{\|p_k\|^2}{c_{\min}(x_k)}}\cdot\frac{-c_i(x_k)}{c_{\min}(x_k)}
		\end{eqnarray*} 
		where $c_{\min}(x_k)$ denotes the minimum constraint margin at the $k$-th iteration:
		$$c_{\min}(x_k):=\min\{-c_1(x_k),-c_2(x_k),\cdots,-c_m(x_k)\}.$$
		Denoting $L_{\max} :=\max\{L_1,L_2,\cdots,L_m\}$, inequality \eqref{thm:AFS-10} indicates that 
		$$ \frac{\|p_k\|^2}{c_{\min}(x_k)} \geq \frac{1}{m L_{\max}}\sum_{i=1}^m\frac{L_i\|p_k\|^2}{-c_i(x_k)}   \geq \frac{C_4/C_3}{mL_{\max}}.$$ 
		Substituting the above two bounds to \eqref{thm:AFS-13} gives the type-2 descent inequality:
		\begin{equation}
			\label{thm:AFS-15}
			-c_i(x_{k+1})\geq -c_i(x_k)\left(1-\frac{m\Lambda  L_{\max}}{C_4^2/C_3}\cdot\frac{-c_i(x_k)}{c_{\min}(x_k)}\cdot\alpha_k\|p_k\|-\frac{\alpha_k^2}{2}\right).
		\end{equation}
		Then for any $k\geq k_5$, let us fix an arbitrary constant greater than 1, say 2. Then we define an index set $\mathcal{I}_{k}$ as 
		$$\mathcal{I}_{k}:=\left\{1\le i\le m: \frac{-c_i(x_k)}{c_{\min}(x_k)}\leq 2\right\}.$$
		Then for any $i\in\mathcal{I}_k$, we apply the type-2 descent inequality \eqref{thm:AFS-15} to yield
		\begin{eqnarray}
			\label{thm:AFS-16}
			-c_i(x_{k+1}) &\geq& -c_i(x_k)\left(1-\frac{2m\Lambda  L_{\max}}{C_4^2/C_3}\cdot\alpha_k\|p_k\|-\frac{\alpha_k^2}{2}\right) \\
			& \geq & c_{\min}(x_k)\left(1-\frac{2m\Lambda  L_{\max}}{C_4^2/C_3}\cdot\alpha_k\|p_k\|-\frac{\alpha_k^2}{2}\right)\quad\mbox{for}\quad k\geq k_5, i\in\mathcal{I}_k.\nonumber
		\end{eqnarray}
		Note that $\alpha_k\to0$ and $\|p_k\|\to0$, we assume w.l.o.g. that the above right hand side factor lies in $(0,1)$. Otherwise, we only need to increase $k_5$ until it is large enough.  For any $i\notin\mathcal{I}_k$, we can apply the type-1 descent inequality \eqref{thm:AFS-14} to yield
		\begin{eqnarray*}
			-c_i(x_{k+1})  \geq  -c_i(x_k)\left(1-\alpha_k-\frac{\alpha_k^2}{2}\right) \geq  c_{\min}(x_k)\cdot 2\left(1-\alpha_k-\frac{\alpha_k^2}{2}\right)
		\end{eqnarray*}
		Note that $\alpha_k\to0$, we can assume w.l.o.g. that 
		$2(1-\alpha_k-\alpha_k^2)\geq 1$ for all $k\geq k_5$. Otherwise it is again sufficient to increase $k_5$ until this inequality holds. Therefore, we have 
		\begin{eqnarray}
			\label{thm:AFS-17}
			-c_i(x_{k+1})  \geq   c_{\min}(x_k) \quad\mbox{for}\quad k\geq k_5, i\notin\mathcal{I}_k.\nonumber
		\end{eqnarray}
		Combining \eqref{thm:AFS-16} and \eqref{thm:AFS-17} gives 
		\begin{equation}
			\label{thm:AFS-18}
			c_{\min}(x_{k+1}) \geq c_{\min}(x_k)\left(1-\frac{2m\Lambda  L_{\max}}{C_4^2/C_3}\cdot\alpha_k\|p_k\|-\frac{\alpha_k^2}{2}\right).
		\end{equation}
		Consequently, we can lower bound $c_{\min}(x^*)$ by the following infinite product 
		$$c_{\min}(x^*)\geq c_{\min}(x_{k_5})\cdot\prod_{k=k_5}^{\infty}\left(1-\frac{2m\Lambda  L_{\max}}{C_4^2/C_3}\cdot\alpha_k\|p_k\|-\frac{\alpha_k^2}{2}\right).$$
		For an infinite product $\prod_{k=0}^{\infty}(1-\beta_k)$ with $\beta_k\in(0,1)$, it is well-known that
		$$\prod_{i=1}^{\infty}(1-\beta_i)>0\qquad \Longleftrightarrow \qquad \sum_{i=1}^{\infty}\beta_i<+\infty.$$
		With this fact, combining \eqref{eqn:summable} and \eqref{thm:AFS-11} gives 
		$$\sum_{k=k_5}^{\infty}\frac{2m\Lambda  L_{\max}}{C_4^2/C_3}\cdot\alpha_k\|p_k\|+\frac{\alpha_k^2}{2} \leq \frac{4m\Lambda  L_{\max}\Delta_f}{LC_4^3/C_3} +\sum_{k=0}^\infty\alpha_k^2<+\infty,$$
		indicating that $c_{\min}(x^*)>0$, which contradicts the basic assumption of this branch that $\cA(x^*)\neq\emptyset$. 
		
		\textbf{Branch 10.} Therefore, \eqref{thm:AFS-fake} will not hold and we must have 
		\begin{equation}
			\label{thm:AFS-19}
			\liminf_{k\to\infty} R_k = 0.
		\end{equation}
		In this case, there exists a subsequence $\KK_2\subseteq\mathbb{Z}_+$ such that $R_k\to0$, $k\in\KK_2$. In this case, we recall \eqref{thm:AFS-lamA} and \eqref{thm:AFS-9} to obtain that 
		\begin{eqnarray}
			\lim_{k\in\KK_2}\lambda_{k,A} &=& -\lim_{k\in\KK_2}\left[J_{k,A}J_{k,A}^\top\right]^{-1}\!\!J_{k,A}\left(\nabla f(x_k) +    J_{k,N}^\top\lambda_{k,N} + R_k\cdot\bar{p}_k + Lp_k\right)\\
			& = & -\left[J_{*,A}J_{*,A}^\top\right]^{-1}\!\!J_{*,A}\nabla f(x_*).
		\end{eqnarray} 
		Define this vector as $\lambda^*_A = [\lambda^*_i]_{i\in\cA(x^*)}$, then we have 
		$\lambda_{A}^*:= -\left[J_{*,A}J_{*,A}^\top\right]^{-1}\!\!J_{*,A}\nabla f(x_*)$. Note that $J_{*,A}^\top\left[J_{*,A}J_{*,A}^\top\right]^{-1}\!\!J_{*,A} = I$, we obtain the following relationship
		$$0 = \nabla f(x^*) + J_{*,A}^\top\lambda_{A}^* = \nabla f(x^*) +\sum_{i\in\cA(x^*)} \lambda_i^*\nabla c_i(x^*).$$
		For $i\notin\cA(x^*)$, it is sufficient to set $\lambda^*_i=0$, which corresponds to the observation \eqref{thm:AFS-9}. This completes the proof of the theorem. 
	\end{proof}


\end{document}